\documentclass{article}
\usepackage{geometry}
\usepackage{amsmath, amsthm, amsfonts, amssymb, enumitem, esint}
\usepackage[hidelinks]{hyperref}

\numberwithin{equation}{section}

\newtheorem*{theorem*}{Theorem}
\newtheorem{theorem}{Theorem}[section]
\newtheorem{lemma}[theorem]{Lemma}
\newtheorem{proposition}[theorem]{Proposition}
\newtheorem{corollary}[theorem]{Corollary}
\theoremstyle{definition}
\newtheorem{definition}[theorem]{Definition}
\theoremstyle{remark}

\DeclareMathOperator{\diam}{diam}
\DeclareMathOperator{\dist}{dist}
\DeclareMathOperator{\e}{e}
\DeclareMathOperator{\Id}{Id}
\DeclareMathOperator{\Imm}{Imm}
\DeclareMathOperator{\Ort}{O}
\DeclareMathOperator*{\osc}{osc}
\DeclareMathOperator{\rank}{rank}
\DeclareMathOperator{\SO}{SO}
\DeclareMathOperator{\lspan}{span}
\DeclareMathOperator{\supp}{supp}

\DeclareMathOperator{\vol}{vol}

\newcommand{\csubset}{\subset\joinrel\subset}

\title{\textbf{Rigidity of codimension-1 isometric immersions in complete manifolds}}
\author{\textbf{Mert Baştuğ}}
\date{}

\begin{document}

\maketitle

\begin{abstract}
    We establish an asymptotic rigidity result for isometric immersions of codimension-1. Specifically, we consider a sequence of immersions from a compact $d$-dimensional manifold into a complete $(d+1)$-dimensional manifold whose elastic energies vanish asymptotically, where the elastic energy quantifies both stretching and bending. We show that such a sequence admits a subsequence converging to an isometric immersion. This extends a result of Alpern, Kupferman, and Maor \cite{AKM2} to the case of complete target manifolds, where the lack of compactness introduces additional analytical difficulties. The proof is based on an approach using local quantitative rigidity estimates, obtained via a reduction to the Euclidean setting. This method avoids the use of Young measures and provides a flexible framework that may be of independent interest.
\end{abstract}

\section{Introduction}

According to Liouville's theorem, if a map $u \in C^1(U; \mathbb{R}^d)$, defined on the open and connected set $U \subset \mathbb{R}^d$, satisfies $Du(x) \in \SO(d)$ for all $x \in U$, then $u$ is, in fact, a rigid motion: $u(x) = Rx + b$ for some $R \in \SO(d)$ and $b \in \mathbb{R}^d$. This classical result has been the starting point for various extensions and refinements in the study of rigidity. Reshetnyak \cite{R2, R1} generalized it to an asymptotic rigidity result for maps of lower regularity: If $u_k \in W^{1, p}(U; \mathbb{R}^d)$ and $\dist(Du, \SO(d)) \to 0$ in $L^p$ for $p \in [1, \infty)$, then, up to a subsequence, there exists $R \in \SO(d)$ such that $Du_k \to R$ in $L^p$. More recently, Friesecke, James and Müller obtained their celebrated quantitative rigidity estimate \cite{FJM} (see also \cite[Section 2.4]{CS}), which states that for $p \in (1, \infty)$ and for every $u \in W^{1, p}(U; \mathbb{R}^d)$ defined on a Lipschitz domain $U \subset \mathbb{R}^d$, there exists $R \in \SO(d)$ such that
\begin{equation} \label{eq:fjm}
    \|Du - R\|_{L^p} \le C \|\dist(Du, \SO(d))\|_{L^p},
\end{equation}
where $C$ depends only on $U$ and $p$. The quantitative rigidity estimate plays a central role in nonlinear elasticity, in particular in the rigorous derivation of lower-dimensional models such as plate and shell theories via $\Gamma$-convergence \cite{FJMM, FJM, FJM2, MS}.

Due to recent interest in non-Euclidean elasticity, rigidity results similar to those stated above have been generalized to the Riemannian setting. Let $(M, g)$ and $(N, h)$ be $d$-dimensional compact, oriented Riemannian manifolds. Kupferman, Maor and Schachar \cite{KMS} showed that if $u_k \in W^{1, p}(M; N)$ and
\[
\int_M \dist_{g, h}^p(d(u_k)_x, \SO((T_xM, g_x), (T_{u_k(x)}N, h_{u_k(x)})) \, d\vol_g(x) \to 0
\]
(see Section \ref{sec:notation} for an explanation of the notation), then, up to a subsequence, $(u_k)$ converges in $W^{1, p}$ to a smooth isometric immersion. Furthermore, Conti, Dolzmann and Müller \cite{CDM} extended the quantitative rigidity estimate \eqref{eq:fjm} to maps in $W^{1, p}(M; M)$. Roughly speaking, they prove that for every $u \in W^{1, p}(M; M)$, there exists an orientation preserving isometry $\phi$ such that the $W^{1, p}$-distance between $u$ and $\phi$ is bounded by the ``elastic energy"
\[
\int_M \dist_g^p(du_x, \SO(T_xM, g_x)) \, d\vol_g(x).
\]
See also \cite[Figure 1]{AKM2} for a summary of the existing results.

Such rigidity statements do not hold for codimension-1 maps, that is when the codomain of the map has 1 dimension more than its domain. For example, for any smooth unit-speed curve $u : \mathbb{R} \rightarrow \mathbb{R}^2$, the derivative $du_x : T_x\mathbb{R} \rightarrow T_{u(x)}\mathbb{R}^2$ is a linear isometry for all $x \in \mathbb{R}$, yet $u'$ is not constant. The reason for the gain in flexibility in higher codimension is that the elastic energy we have considered so far does not restrict bending, which is the source of oscillations. To remedy this issue, we augment the elastic energy by adding a term to account for the bending. Suppose $M$ and $N$ are $d$ and $(d + 1)$-dimensional, respectively. Given an immersion $u : M \rightarrow N$ and $x \in M$, let $\nu_u(x) \in T_{u(x)}N$ be the unit normal to $du_x(T_xM)$ that is consistent with the orientation induced by $du_x$. The variation of $\nu_u$ in the ambient space $N$ is measured by a map $S_u : TM \rightarrow TM$ called the shape operator induced by $u$ (we give a more precise definition in Section \ref{sec:sobolev_immersions}). Fix $p \in [1, \infty)$. To each $u$ we assign the sum of a stretching and a bending energy:
\begin{equation*}
    E(u) := \int_M \dist_{g, h}^p(du_x, \Ort((T_xM, g_x), (T_{u(x)}N, h_{u(x)}))) \, d\vol_g(x) + \int_M |du_x \circ (S_u(x) - S(x))|_{g, h}^p \, d\vol_g(x),
\end{equation*}
where $S : TM \rightarrow TM$ is a reference shape operator on $M$. The energy is well-defined on the space of Sobolev immersions given by
\[
\Imm_p(M; N) := \{u \in W^{1, p}(M; N) : \rank du_x = d \text{ for a.e. } x \in M, \, \nu_u \in W^{1, p}(M; TN)\}.
\]
In \cite[Theorem 1.1]{AKM2}, Alpern, Kupferman and Maor prove that if $M$ and $N$ are compact and $E(u_k) \to 0$ for $u_k \in \Imm_p(M; N)$, then there exists a smooth isometric immersion $u \in \Imm_p(M; N)$ such that, up to a subsequence, $(u_k)$ and $(\nu_{u_k})$ converge in $W^{1, p}(M; N)$ to $u$ and $\nu_u$, respectively, and $S_u = S$. Furthermore, the result in \cite[Theorem 1.1]{AKM2} extends a similar result proved by the same authors in \cite{AKM}, where the additional assumption of constant sectional curvature in $N$ was imposed.

The purpose of this paper is to extend the results of Alpern, Kupferman and Maor to the case when $N$ is a complete, not necessarily compact manifold. Our main result can be summarized as follows:

\begin{theorem*}
        Assume $p \in (1, \infty)$. Let $(M, g)$ and $(N, h)$ be oriented, connected, complete Riemannian manifolds of dimensions $d$ and $d + 1$, respectively. Let $M$ be compact. Assume $(u_k) \subset \Imm_p(M; N)$ satisfies $E(u_k) \to 0$ and
        \begin{equation*}
            \limsup_{k \to \infty} \int_M d_h^p(u_k, q_0) \, d\vol_g < \infty
        \end{equation*}
        for some $q_0 \in N$. Then, up to a subsequence, $u_k$ converges in $W^{1, p}$ to an isometric immersion $u \in \Imm_p(M; N)$. If $|S|_g \in L^{p'}(M)$, where $p'$ is the Hölder conjugate of $p$, then $S = S_u$.
\end{theorem*}

See Theorem \ref{thm:asymptotic_rigidity} for a more detailed statement. In comparison to the Young measure approach employed in \cite{AKM2}, our proof is based on local quantitative rigidity estimates for Sobolev immersions, which we derive by a reduction to the Euclidean rigidity estimate \eqref{eq:fjm}. The rigidity estimates are of independent interest.

The main idea behind our rigidity estimate can be explained easily for a map $u \in \Imm_p(M; \mathbb{R}^{d + 1})$. Consider a sufficiently small chart $(W, \psi)$ on $M$ on which the metric $g$ is close to being constant. If the unit normal $\nu_u$ is approximately constant in $W$, then $u(W)$ is close to a $d$-dimensional hyperplane. If we denote by $P$ the projection onto this hyperplane, then the equidimensional map $v := P \circ u \circ \psi^{-1}$ is a close approximation to $u$. Since the metric is almost constant, we can apply the FJM estimate \eqref{eq:fjm} to $v$ to measure the deviation of $du$ from a rotation in terms of its stretching energy in $W$. Our assumption that the normal vector $\nu_u$ does not vary too much in small regions is not true in general. However, thanks to the Poincaré inequality and the control on the variation of $\nu_u$ provided by the bending energy, these regions cannot be too big. 

In order to extend these considerations to Sobolev immersions in $\Imm_p(M; N)$, we could embed $N$ isometrically into a large Euclidean space. However, this approach has the drawback that the results depend on the embedding of $N$. This turns out not to be a problem when $N$ is a compact manifold. Indeed, in a companion paper \cite{B}, we follow this approach to deliver a simpler proof for compact target manifolds. Instead, in the current work, we localize in the codomain as well. However, a major difficulty arises due to the fact that an arbitrary Sobolev map in $W^{1, p}(M; N)$ is, in general, discontinuous. Hence, the image of a small region in $M$ might never be contained in a single chart on $N$. To deal with this problem we follow \cite{CVS} and use extended charts on $N$, that is, globally defined maps that take $N$ into $\mathbb{R}^{d + 1}$ that are diffeomorphisms only in a certain region of $N$ (see also \cite[Lemma 1.6]{CVS}). For the final result, we refer the reader to Theorem \ref{thm:local_quantitative_rigidity}.

We believe that our methods are quite flexible and elementary apart from the reference to the Euclidean quantitative rigidity estimate. The approach can also be extended easily to higher codimensions. However, we note that, in contrast to \cite{AKM, AKM2}, we do not prove the smoothness of the limiting isometric immersion.

The paper is organized as follows. In Section \ref{sec:preliminaries}, we introduce the necessary definitions on Riemannian structures on vector bundles and weakly differentiable maps between manifolds. Since we do not embed $N$ isometrically into Euclidean space, we use the intrinsic definition of manifold-valued Sobolev maps throughout the work, which is based on the notion of colocal weak differentiability (see Definition \ref{def:cwd}) introduced by Convent and Van Schaftingen \cite{CVS}. In Section \ref{sec:quantitative}, we define $\varepsilon$-isometric charts and extensions of charts via cut-offs (see Definitions \ref{def:almost_isometric_chart} and \ref{def:extension}). In the end, we prove the local quantitative rigidity theorem (Theorem \ref{thm:local_quantitative_rigidity}). In Section \ref{sec:asymptotic}, we state our main result (Theorem \ref{thm:asymptotic_rigidity}) and prove some preliminary results. Finally, Section \ref{sec:proof_of_asymptotic_rigidity} is devoted to the proof of Theorem \ref{thm:asymptotic_rigidity}.

\section{Preliminaries} \label{sec:preliminaries}

In this section, we give a brief but mostly self-contained introduction to Riemannian structures on vector bundles and weakly differentiable maps between manifolds, which already appear in the approach by Alpern, Kupferman and Maor \cite{AKM, AKM2}. After a subsection on notation, we recall the definition of a metric connection on a vector bundle $(E, \pi, N)$ and define the connector operator, which allows us to construct an isomorphism between $TE$ and $TN \oplus E$. With the help of this isomorphism, we define the Sasaki metric on the tangent bundle $TE$. For $E = T^*M \otimes TN$, the Sasaki metric allows us to define a distance function between differentials of maps from $M$ to $N$. At the end, we also express the connector operator in coordinates, since this will be useful in later estimates. In the last subsection, we give a brief account of colocally weakly differentiable functions introduced by Convent and Van Schaftingen \cite{CVS}. Later, we use these functions to generalize Sobolev spaces to Riemannian manifolds. Finally, we define Sobolev immersions based on the paper \cite{AKM}.

\subsection{Notation} \label{sec:notation}

Let $V$ be a vector space endowed with a constant metric $g_0$. For $v, v' \in V$, we write $(v, v')_{g_0}$ in place of $g_0(v, v')$ and set $|v|_{g_0} := \sqrt{(v, v)_{g_0}}$. Let $W$ be another vector space endowed with a constant metric $h_0$. We denote by $\mathcal{L}(V, W)$ the space of linear maps from $V$ to $W$. The subset of isometries in $\mathcal{L}(V, W)$ is denoted by $\Ort((V, g_0), (W, h_0))$, while the orientation preserving ones are written as $\SO((V, g_0), (W, h_0))$. 

Given an orthonormal basis $(v_1, \dots, v_d)$ of $V$, We define the Frobenius norm of $T \in \mathcal{L}(V, W)$ by
\[
|T|_{g_0, h_0} := \left(\sum_{i = 1}^d |Tv_i|_{h_0}^2\right)^\frac{1}{2}.
\]
This definition does not depend on the choice of orthonormal basis. In the case $V = W$ and $g_0 = h_0$, we abbreviate the notation to $|T|_{g_0}$. For subsets $\mathcal{T}, \mathcal{S} \subset \mathcal{L}(V, W)$, their distance is defined by
\[
\dist_{g_0, h_0}(\mathcal{T}, \mathcal{S}) := \inf \{|T - S|_{g_0, h_0} : T \in \mathcal{T}, S \in \mathcal{S}\}.
\]
We write $\e_d$ for the Euclidean metric on $\mathbb{R}^d$. Whenever all spaces involved are Euclidean, we omit the metric from the notation and simply write $(v, v')$, $|T|$, $\dist(\mathcal{T}, \mathcal{S})$ etc. For two constant metrics $g_0$, $g_0'$ on $\mathbb{R}^d$, we define
\[
|g_0 - g_0'| := \left(\sum_{i, j = 1}^d |(e_i, e_j)_{g_0} - (e_i, e_j)_{g'}|^2\right)^\frac{1}{2},
\]
where $(e_1, \dots, e_d)$ is the standard basis. If $h_0$ is a constant metric on $\mathbb{R}^n$, we use the shorthand $\Ort(g_0, h_0)$ and $\SO(g_0, h_0)$ for $\Ort((\mathbb{R}^d, g_0), (\mathbb{R}^n, h_0))$ and $\SO((\mathbb{R}^d, g_0), (\mathbb{R}^n, h_0))$, respectively. If $g_0 = h_0$, we further abbreviate to $\Ort(g_0)$ and $\SO(g_0)$.

Let $(M, g)$ and $(N, h)$ be Riemannian manifolds. For a smooth map $f : M \rightarrow N$, we denote its differential by $df : TM \rightarrow TN$. In the special case $N = \mathbb{R}^n$, we identify $T\mathbb{R}^n$ with $\mathbb{R}^{2n}$. Thus, for every $x \in M$, there exists a map $Df(x) : T_xM \rightarrow \mathbb{R}^d$ such that
\[
df_x(v) = (f(x), Df(x)v).
\]
If $f$ is only weakly differentiable, the symbols $df$ and $Df$ are used the same sense as in the smooth case. Finally, we denote by $\vol_g$ the volume form on $M$.

\subsection{Riemannian structures on vector bundles}

We recall the definition of a connection on a vector bundle. For details, we refer the reader to \cite[Chapter 12]{LJ} and \cite[Chapters 4,5]{L}. Let $(N, h)$ be an $n$-dimensional Riemannian manifold and let $(E, \pi, N)$ be a smooth vector bundle of rank $k$ equipped with a bundle metric, that is, for every $q \in N$ there exists an inner product $\kappa_q$ on $E_q = \pi^{-1}(q)$ that depends smoothly on $q$. We denote the set of all sections of $E$ by $\Gamma(E)$. A map $\nabla : \Gamma(TN) \times \Gamma(E) \rightarrow \Gamma(E)$, written $(X, Y) \mapsto \nabla_X Y$, is called a connection if
\begin{enumerate}
    \item For all $X_1, X_2 \in \Gamma(TN)$, $Y \in \Gamma(E)$, and $f_1, f_2 \in C^\infty(N)$,
    \[
    \nabla_{f_1 X_1 + f_2 X_2} Y = f_1 \nabla_{X_1} Y + f_2 \nabla_{X_2} Y.
    \]
    \item For all $X \in \Gamma(TN)$, $Y_1, Y_2 \in \Gamma(E)$ and $a_1, a_2 \in \mathbb{R}$,
    \[
    \nabla_X (a_1 Y_1 + a_2 Y_2) = a_1 \nabla_X Y_1 + a_2 \nabla_X Y_2.
    \]
    \item For all $X \in \Gamma(TN)$, $Y \in \Gamma(E)$ and $f \in C^\infty(N)$,
    \[
    \nabla_X (f Y) = f \nabla_X Y + X(f) Y.
    \]
\end{enumerate}
Moreover, $\nabla$ is called metric if for all $X \in \Gamma(TN)$, $Y_1, Y_2 \in \Gamma(E)$,
\[
X((Y_1, Y_2)_\kappa) = (\nabla_X Y_1, Y_2)_\kappa + (Y_1, \nabla_X Y_2)_\kappa.
\]
Given a metric connection $\nabla$ on $E$, we sketch the construction of a natural Riemannian metric on $E$ called the Sasaki metric. The interested reader can refer to \cite[Chapter 3, Ex. 2]{DC}, \cite[Chapter 2.B.6]{GHL} or \cite{S} for more details. We start with a decomposition of the tangent bundle $TE$. Let $e \in E_q$ and let $\alpha: (-\varepsilon, \varepsilon) \rightarrow E$ be a smooth curve with $\alpha(0) = e$. We set $\gamma := \pi \circ \alpha$, and we view $\alpha$ as a vector field along $\gamma$. Intuitively, the local behavior of $\alpha$ at $0$ is determined by the two vectors $\gamma'(0)$ and $\frac{D\alpha}{dt}(0)$, where $\frac{D}{dt}$ is the covariant derivative determined by $\nabla$. Formally, this is reflected in the fact that if $\beta : (-\varepsilon, \varepsilon) \rightarrow E$ is another smooth curve with $\beta(0) = e$ and $\eta = \pi \circ \beta$, then
\begin{equation} \label{eq:curve_local_behavior}
    \alpha'(0) = \beta'(0) \iff \left(\gamma'(0), \frac{D\alpha}{dt}(0)\right) = \left(\eta'(0), \frac{D\beta}{dt}(0)\right).
\end{equation}
This observation yields a natural decomposition of $TE$ as we now show.

Let $\xi \in TE$. Then $\xi = \alpha'(0)$ for a smooth curve $\alpha : (-\varepsilon, \varepsilon) \rightarrow E$. We define the so called connector operator (see also \cite[Chapter 4]{M})
\begin{align} \label{eq:connector_operator}
    K_E : TE &\rightarrow E, \\
    \xi &\mapsto \frac{D\alpha}{dt}(0). \nonumber
\end{align}
By \eqref{eq:curve_local_behavior}, the connector operator is well-defined. We remark that the connector operator is simply another way of expressing the covariant derivative. If $X \in \Gamma(TN)$ and $Y \in \Gamma(E)$, then $K_E \circ dY \circ X = \nabla_X Y$ as the reader can easily verify. 

If we denote by $TN \oplus E$ the Whitney sum of $TN$ and $E$, then it follows from \eqref{eq:curve_local_behavior} and the relation $\gamma'(0) = d\pi(\xi)$ that the following map is an isomorphism:
\begin{align*}
    TE &\rightarrow TN \oplus E, \\
    \xi &\mapsto (d\pi(\xi), K_E(\xi)).
\end{align*}
Identifying $TE$ with $TN \oplus E$ according to the isomorphism and declaring $TN$ orthogonal to $E$ yields the inner product
\begin{equation} \label{eq:sasaki_metric}
    (\xi_1, \xi_2) \mapsto (d\pi(\xi_1), d\pi(\xi_2))_h + (K_E(\xi_1), K_E(\xi_2))_\kappa, \quad \xi_1, \xi_2 \in T_eE.
\end{equation}
As a result, we obtain a Riemannian metric on $E$ that we call the Sasaki metric. We denote it by $\sigma$. The Sasaki metric enables us to define a distance on the vector bundle $E$. The corresponding Riemannian distance function is called the Sasaki distance.

In the following proposition, we give an explicit characterization of the Sasaki distance \cite[Section 2]{S}.
\begin{proposition} \label{pr:sasaki_distance}
    Let $N$ be connected. Let $q_1, q_2 \in N$ and let $e_1 \in E_{q_1}$, $e_2 \in E_{q_2}$. Then
    \begin{equation} \label{eq:sasaki_distance}
        d_\sigma(e_1, e_2)^2 = \inf_\gamma |e_1 - P^\gamma(e_2)|_\kappa^2 + \left(\int_0^1 |\gamma'|_h\, dt\right)^2,
    \end{equation}
    where the infimum is taken over all piecewise regular curves $\gamma: [0, 1] \rightarrow N$ with $\gamma(0) = q_1, \gamma(1) = q_2$, and $P^\gamma : E_{q_2} \rightarrow E_{q_1}$ is the parallel transport along $\gamma$ with respect to $\nabla$. 
\end{proposition}

\begin{proof}
    Let $\alpha: [0, 1] \rightarrow E$ be a piecewise regular curve with $\alpha(0) = e_1$ and $\alpha(1) = e_2$. Set $\gamma := \pi \circ \alpha$. Let $(X_i(t))_{i = 1}^k$ be an orthonormal basis of $E_{\gamma(t)}$ for $t \in [0, 1]$, and assume that $X_i$ is parallel along $\gamma$, that is, $\frac{DX_i}{dt} = 0$. Then there exists a function $a: [0, 1] \rightarrow \mathbb{R}^k$ such that
    \[
    \alpha(t) = \sum_{i = 1}^k a_i(t) X_i(t).
    \]
    Since $\frac{D X_i}{dt} = 0$, we get
    \[
    \frac{D\alpha}{dt}(t) = \sum_{i = 1}^k a'_i(t)X_i(t).
    \]
    Therefore,
    \[
    |\alpha'(t)|_\sigma^2 = |\gamma'(t)|_h^2 + \left|\sum_{i = 1}^k a_i'(t)X_i(t)\right|_\kappa^2 = |\gamma'(t)|_h^2 + |a'(t)|^2.
    \]
    We set
    \[
    l_1 = \int_0^1 |\gamma'|_h \, dt , \quad l_2 = \int_0^1 |a'| \, dt.
    \]
    By the concavity of the square root function, we obtain
    \[
    \int_0^1 |\alpha'|_\sigma \, dt = \int_0^1 (|\gamma'|_h^2 + |a'|^2)^\frac{1}{2} \, dt \ge \frac{l_1}{(l_1^2 + l_2^2)^\frac{1}{2}} \int_0^1 |\gamma'|_h \, dt + \frac{l_2}{(l_1^2 + l_2^2)^\frac{1}{2}} \int_0^1 |a'| \, dt = (l_1^2 + l_2^2)^\frac{1}{2}.
    \]
    Notice that $l_2 \ge |a(0) - a(1)| =  |e_1 - P^\gamma(e_2)|_\kappa$. This proves that the left-hand side is greater than or equal to the right-hand side in \eqref{eq:sasaki_distance}. To show the reverse inequality let $\gamma: [0, 1] \rightarrow N$ be piecewise regular curve with $\gamma(0) = q_1, \gamma(1) = q_2$. Define $a, b \in \mathbb{R}^k$ and $\alpha : [0, 1] \to E$ by
    \[
    e_1 = \sum_{i = 1}^k a_i X_i(0), \quad e_2 = \sum_{i = 1}^k b_i X_i(1), \quad \alpha(t) := \sum_{i = 1}^k (a_i + t(b_i - a_i)) X_i(t).
    \]
    Then $\alpha(0) = e_1$, $\alpha(1) = e_2$ and
    \[
    d_\sigma(e_1, e_2)^2 \le \left(\int_0^1 |\alpha'|_\sigma \, dt\right)^2 \le \int_0^1 |\alpha'|_\sigma^2 \, dt \le \int_0^1 |a - b|^2 + |\gamma'|_h^2 \, dt = |e_1 - P^\gamma e_2|_\kappa^2 + \int_0^1 |\gamma'|_h \, dt.
    \]
\end{proof}

We will use the following simple facts throughout the article. The zero vector in $E_q$ is denoted by $0_q$.

\begin{proposition} \label{pr:triangle_inequality}
     Let $q_1, q_2 \in N$ and let $e_0, e_1 \in E_{q_1}$, $e_2 \in E_{q_2}$. Then
    \begin{gather*}
        d_\sigma(e_0, e_1) = |e_0 - e_1|_\kappa, \\
        d_\sigma(0_{q_1}, 0_{q_2}) = d_h(q_1, q_2), \\
        d_\sigma(e_1, e_2) \le |e_1|_\kappa + |e_2|_\kappa + d_h(q_1, q_2).
    \end{gather*}
\end{proposition}

\begin{proof}
    The first and second equalities follow immediately from \ref{eq:sasaki_distance}. To prove the third one we use the triangle inequality:
    \[
    d_\sigma(e_1, e_2) \le d_\sigma(e_1, 0_{q_1}) + d_\sigma(0_{q_1}, 0_{q_2}) + d_\sigma(0_{q_2}, e_2) = |e_1|_\kappa + |e_2|_\kappa + d_h(q_1, q_2).
    \]
\end{proof}

Let $(M, g)$ be a connected $d$-dimensional Riemannian manifold. In this paper, we employ the Sasaki metric in order to define a distance between differentials of maps from $M$ to $N$. To this end, we set $E := T^*M \otimes TN$, and for $L \in T_x^*M \otimes T_qN \subset T^*M \otimes TN$, we define $\pi(L) := (x, q)$. Then $(E, \pi, M \times N)$ is a smooth vector bundle. Furthermore, we can endow each fiber $T_p^*M \otimes T_qN$ with the inner product $g_x^* \otimes h_q$. In order to define a metric connection on $T^*M \otimes TN$, we let $\nabla^g$ and $\nabla^h$ be the Levi-Civita connections on $M$ and $N$, respectively. Then there exists a unique metric connection $\nabla$ satisfying
\begin{equation} \label{eq:tensor_bundle_connection}
    \nabla_{(Y, Z)} (\omega \otimes X) = (\nabla_Y^g \omega) \otimes X + \omega \otimes (\nabla_Z^h X)
\end{equation}
for all $\omega \in \Gamma(T^*M)$, $X \in \Gamma(TN)$, $Y \in \Gamma(TM)$ and $Z \in \Gamma(TN)$.

Using Proposition \ref{pr:sasaki_distance}, we can also derive an explicit characterization for the Sasaki distance on $T^*M \otimes TN$. If $\gamma : [0, 1] \rightarrow N$ is a piecewise regular curve with $\gamma(0) = q_1$ and $\gamma(1) = q_2$, then we denote the parallel transport map along $\gamma$ with respect to $\nabla^h$ by $P_{TN}^\gamma : T_{q_2}N \rightarrow T_{q_1}N$.

\begin{corollary} \label{co:maps_sasaki_distance}
     Let $x \in M$, $q_1, q_2 \in N$ and let $L_1 \in T_x^*M \otimes T_{q_1}N$, $L_2 \in T_x^*M \otimes T_{q_2}N$. Then
    \begin{equation} \label{eq:maps_sasaki_distance}
        d_\sigma(L_1, L_2)^2 = \inf_\gamma |L_1 - L_\gamma|_{g, h}^2 + \left(\int_0^1 |\gamma'|_h \, dt\right)^2,
    \end{equation}
    where the infimum is taken over all piecewise regular curves $\gamma: [0, 1] \rightarrow N$ with $\gamma(0) = q_1, \gamma(1) = q_2$ and
    \begin{equation} \label{eq:linear_map_transport}
        L_\gamma v = P_{TN}^\gamma(L_2v) \quad \text{ for all } v \in T_xM.
    \end{equation}
    
\end{corollary}

Let $P_E^\gamma : E_{(x, q_2)} \rightarrow E_{(x, q_1)}$ be the parallel transport map along $(x, \gamma)$ with respect to $\nabla$. The proof of the previous corollary depends on the following relation between $P_E^\gamma$ and $P_{TN}^\gamma$.

\begin{lemma} \label{lm:map_parallel_transport}
    Let $L \in T_x^*M \otimes T_{q_2}N$. Then for all $v \in T_xM$, we have 
    \[
    P_E^{\gamma}(L)v = P_{TN}^\gamma(Lv).
    \]
\end{lemma}

\begin{proof}
    The statement becomes a tautology once we write out all the definitions. Let $(v_j)_{j = 1}^m$ be a basis of $T_xM$ and let $(X_i(t))_{i = 1}^n$ be a basis of $T_{\gamma(t)}N$ for $t \in [0, 1]$. Assume that $X_i$ is parallel along $\gamma$. Denote the dual basis of $(v_j)$ by $(\omega_j)$. Then 
    \[
    (\omega_j \otimes X_i(t))_{i = 1, \dots, n}^{j = 1, \dots, m}
    \]
    is a basis for $T_x^*M \otimes T_{\gamma(t)}N$. We recall that $\omega_j \otimes X_i(t)$ is identified with the linear map taking $v_j$ to $X_i(t)$ and $v_k$ to $0$ for $k \neq j$. Then $L$ is identified with
    \[
    \sum_{i = 1}^n \sum_{j = 1}^m a_{ij} \, \omega_j \otimes X_i(1),
    \]
    where $a_{ij}$ satisfy
    \[
    Lv_j = \sum_{i = 1}^n a_{ij} X_i(1).
    \]
    By \eqref{eq:tensor_bundle_connection},
    \[
    \frac{D}{dt}(\omega_j \otimes X_i) = \frac{D\omega_j}{dt} \otimes X_i + \omega_j \otimes \frac{DX_i}{dt} = 0,
    \]
    so that $P_E^\gamma(L)$ is identified with
    \[
    \sum_{i = 1}^n \sum_{j = 1}^m a_{ij} \, \omega_j \otimes X_i(0).
    \]
    This implies
    \[
    P_E^\gamma(L)v_j = \sum_{i = 1}^n a_{ij} X_i(0) = P_{TN}^\gamma(Lv_j) \quad \text{for all } j.
    \]
    \end{proof}

\subsubsection{The connector operator in coordinates}

Let $(N, h)$ be an $n$-dimensional Riemannian manifold and let $\nabla^h$ be its Levi-Civita connection. In this section, we give an explicit representation of the connector operator in coordinates. Let $(U, \varphi)$ be a chart on $N$ and denote the local coordinates on $U$ by $(y_1, \dots, y_n)$. Set
\[
X_i(q) := \frac{\partial}{\partial y_i}\bigg|_q, \quad q \in U.
\]
The Christoffel symbols $\Gamma_{i, j}^k : U \rightarrow \mathbb{R}$ are given by
\[
\nabla^h_{X_i} X_j = \sum_{k = 1}^n \Gamma_{i, j}^k X_k.
\]
We can define coordinates on $TU \subset TN$ through the map
\[
(y, w) \in U \times \mathbb{R}^n\mapsto \sum_{i = 1}^n w_i X_i(\varphi^{-1}(y)) \in TU.
\]
Hence, we can represent any $\xi \in T_{(q, v)}TN$ with $q \in U$ in the form
\begin{equation} \label{eq:coordinates_double_tangent_bundle}
    \sum_{i = 1}^n Y_i \frac{\partial}{\partial y_i}\bigg|_{(q, v)} + \sum_{i = 1}^n W_i \frac{\partial}{\partial w_i}\bigg|_{(q, v)}
\end{equation}
for some $Y, W \in \mathbb{R}^n$.

Fix $(q, v) \in TN$ with $q \in U$ and fix $\xi \in T_{(q, v)}TN$. Let $\alpha : (-\varepsilon, \varepsilon) \rightarrow TN$ be a smooth curve with $\alpha(0) = (q, v)$ and $\alpha'(0) = \xi$. Let $\pi_N : TN \rightarrow N$ be the canonical projection and set $\gamma := \pi_{TN} \circ \alpha$.
If $(\tilde{y}(t), \tilde{w}(t))$ are the local coordinates for $\alpha(t)$, then (\cite[Page 56]{DC})
\begin{gather*}
    \frac{D\alpha}{dt}(0) = \sum_{k = 1}^n \left(\tilde{w}_k'(0) + \sum_{i, j = 1}^n \tilde{w}_j(0) \tilde{y}_i'(0) \Gamma_{i, j}^k(q) \right)X_k(q), \\
    \gamma'(0) = \sum_{i = 1}^n \tilde{y}_i'(0) X_i(q).
\end{gather*}
Furthermore, comparing \eqref{eq:coordinates_double_tangent_bundle} with the representation
\[
\alpha'(0) = \sum_{i = 1}^n \tilde{y}_i'(0) \frac{\partial}{\partial y_i}\bigg|_{(q, v)} + \sum_{i = 1}^n \tilde{w}_i'(0) \frac{\partial}{\partial w_i}\bigg|_{(q, v)}
\]
gives $\tilde{y}'(0) = Y$ and $w'(0) = W$. Hence, if we set $w := \tilde{w}(0)$, then, by \eqref{eq:connector_operator}, we obtain
\begin{equation} \label{eq:connector_operator_coordinates}
    K_{TN}(\xi) = \sum_{k = 1}^n \left(W_k + \sum_{i, j = 1}^n w_j Y_i \Gamma_{i, j}^k(q)\right) X_k(q),
\end{equation}
We also note that
\begin{equation} \label{eq:bundle_projection_coordinates}
    d\pi_N(\xi) = \sum_{i = 1}^n Y_i X_i(q),
\end{equation}
where $\pi_N : TN \rightarrow N$ is the bundle projection.

If $N = \mathbb{R}^n$, then we can consider the globally defined chart $(\mathbb{R}^n, \Id_{\mathbb{R}^n})$ and identify $TT\mathbb{R}^n$ with $\mathbb{R}^{4n}$ using the map
\begin{equation} \label{eq:double_tangent_bundle}
    (y, w, Y, W) \mapsto \sum_{i = 1}^n Y_i \frac{\partial}{\partial y_i}\bigg|_{(y, w)} + \sum_{i = 1}^n W_i \frac{\partial}{\partial w_i}\bigg|_{(y, w)}   
\end{equation}
In this case the connector operator is given by the map
\begin{equation} \label{eq:connector_operator_euclidean}
    K_{T\mathbb{R}^n}(y, w, Y, W) =  \sum_{k = 1}^n W_k X_k(y).
\end{equation}
Similarly, we have
\begin{equation} \label{eq:euclidean_bundle_projection}
    d\pi_{\mathbb{R}^n}(y, w, Y, W) = \sum_{i = 1}^n Y_i X_i(y).
\end{equation}

\subsection{Weakly differentiable maps between manifolds} \label{sec:weakly_differentiable}

In this section we give an intrinsic definition of weak differentiability for maps between manifolds using the notions developed by Convent and Van Schaftingen \cite{CVS}. We start the section by recalling the corresponding definitions for real valued functions on manifolds. Throughout the paper, we do not identify functions that agree almost everywhere.

Let $M$ be a smooth manifold. We say that $A \subset M$ is negligible, if for every $x \in A$, there exists a chart $(W, \psi)$ with $x \in W$ such that $|\psi(W \cap A)| = 0$. If a property holds outside of a negligible subset of $M$, then we say that it holds almost everywhere in $M$. We call $u : M \rightarrow \mathbb{R}$ weakly differentiable if for every chart $(W, \psi)$ the function $u \circ \psi^{-1}$ is weakly differentiable. We can define a local weak derivative $du :TW \rightarrow T\mathbb{R}$ by setting $du := d(u \circ \psi^{-1}) \circ d\psi$. If $(\tilde{W}, \tilde{\psi})$ is another chart, then the chain rule for weakly differentiable functions yields
\[
d(u \circ \psi^{-1})_{\psi(x)} \circ d\psi_x = d(u \circ \tilde{\psi}^{-1})_{\tilde{\psi}(x)} \circ d\tilde{\psi}_x \quad \text{for a.e. } x \in W \cap \tilde{W}.
\]
Hence, we can glue the local weak derivatives together to obtain a map $du : TM \rightarrow T\mathbb{R}$.

Let $N$ be a smooth manifold and let $u : M \rightarrow N$. If $(U, \varphi)$ is a chart on $N$, then the domain of $\varphi \circ u$ is not an open set in general. Hence, it does not make sense to define the weak differentiability of $u$ in terms of its compositions with charts on $N$. An alternative definition was provided by Convent and Van Schaftingen \cite{CVS}. Let $\pi_M : TM \rightarrow M$ and $\pi_N : TN \rightarrow N$ be the canonical projections. The idea is to consider compositions of $u$ with compactly supported smooth functions.

\begin{definition}[{\cite[Definition 1.1.]{CVS}}] \label{def:cwd}
    A map $u : M \rightarrow N$ is called colocally weakly differentiable (c.w.d.) if $f \circ u$ is weakly differentiable for all $f \in C_c^1(N)$. We call $du : TM \rightarrow TN$ a colocal weak derivative of $u$ if $\pi_N \circ du = u \circ \pi_M$ and
    \[
    d(f \circ u)_x = df_{u(x)} \circ du_x \quad \text{for a.e. } x \in M.
    \]
\end{definition}

All colocally weakly differentiable functions possess a unique colocal weak derivative.

\begin{theorem} [{\cite[Proposition 1.5]{CVS}}] \label{th:colocal_derivative_existence}
    If $u : M \rightarrow N$ is c.w.d., then it has a unique colocal weak derivative $du : TM \rightarrow TN$. Moreover, if $f \in C^1(N)$ and $f \circ u$ is weakly differentiable, then
    \[
    d(f \circ u)_p = df_{u(x)} \circ du_x\quad \text{for a.e. } x \in M.
    \]
\end{theorem}

Let $\tilde{N}$ be a smooth manifold. The next two propositions state that compositions and Cartesian products of c.w.d. maps can be computed in the expected manner.

\begin{proposition} [{\cite[Proposition 1.8]{CVS}}] \label{pr:chain_rule}
    Let $u : M \rightarrow N$ be c.w.d. and let $v \in C^1(N; \tilde{N})$. If $v \circ u$ is c.w.d., then
    \[
    d(v \circ u)_x = dv_{u(x)} \circ du_x \quad \text{for a.e. } x \in M.
    \]
\end{proposition}

\begin{proposition} [{\cite[Proposition 1.10]{CVS}}] \label{pr:product_map_derivative}
    If $u_1 : M \rightarrow N$ and $u_2 : M \rightarrow \tilde{N}$ are c.w.d., then so is $u := (u_1, u_2) : M \rightarrow N \times \tilde{N}$, and
    \[
    du_x = (d(u_1)_x, d(u_2)_x) \quad \text{for a.e. } x \in M.
    \]
\end{proposition}

\subsubsection{Homogeneous Sobolev spaces}

Let $(M, g)$ and $(N, h)$ be connected Riemannian manifolds. The Riemannian metric allows us to generalize $L^p$-integrability to maps between manifolds.

\begin{definition}
    Let $1 < p < \infty$. We say that a c.w.d. map $u : M \rightarrow N$ belongs to the homogeneous Sobolev space $\dot{W}^{1, p}(M; N)$ if $|du|_{g, h} \in L^p(M)$.
\end{definition}

Let $\sigma$ denote the Sasaki metric on $(T^*M \otimes TN, \pi, M \times N)$. We define
\[
d_{\sigma, p}(u, v) := \left(\int_M d_\sigma(du_x, dv_x)^p \, d\vol_g(x)\right)^\frac{1}{p}, \quad u, v \in \dot{W}^{1, p}(M; N).
\]
If $(u_k)_{k \in \mathbb{N}} \in \dot{W}^{1, p}(M; N)$, then we say that $(u_k)$ converges to $u$ in $\dot{W}^{1, p}(M; N)$ if $d_{\sigma, p}(u_k, u) \to 0$. 

\begin{proposition} [{\cite[Proposition 4.2]{CVS}}] \label{pr:completeness_of_sobolev_space}
    If $N$ is complete, then the metric space $(\dot{W}^{1, p}(M; N), d_{\sigma, p})$ is complete.
\end{proposition}

For homogeneous Sobolev spaces, we can prove a compactness result similar to the Rellich-Kondrachov theorem in the Euclidean setting. To this end, we require two preliminary results that are of independent interest as well. The following lemma and its proof is adapted from \cite[Proposition 3.8]{CVS}.

\begin{lemma} \label{lm:closure_under_pointwise_convergence}
    Let $(u_k) \subset \dot{W}^{1, p}(M; N)$ be a sequence. Assume that there exists $C > 0 $ such that
    \begin{equation} \label{eq:completeness_of_cwd_functions}
        \int_M |du_k|_{g, h}^p \, d\vol_g \le C \quad \text{for all } k \in \mathbb{N}.
    \end{equation}
    Furthermore, assume that $u_k \to u$ pointwise a.e. in $M$. Then $u \in \dot{W}^{1, p}(M; N)$ and
    \[
    \int_M |du|_{g, h}^p \, d\vol_g \le \liminf_{k \to \infty} \int_M |du_k|_{g, h}^p \, d\vol_g.
    \]
\end{lemma}

\begin{proof}
    Let $f \in C_c^1(N)$. We show that $f \circ u$ is weakly differentiable on $M$. Let $(W, \psi)$ be a chart on $M$ and set $\tilde{u}_k := f \circ u_k \circ \psi^{-1}$, $\tilde{u} := f \circ u \circ \psi^{-1}$. Then $\tilde{u}_k \in W_\text{loc}^{1, p}(W)$. Since $f$ is bounded and $u_k \to u$ pointwise a.e. in $M$, the dominated convergence theorem implies that $\tilde{u}_k \to \tilde{u}$ in $L^p(W)$. In fact, more generally, we have $f \circ u_k \to f \circ u$ in $L^p(M)$. As $W_\text{loc}^{1, p}(W)$ is closed under weak convergence, it follows from \eqref{eq:completeness_of_cwd_functions} that $\tilde{u} \in W_\text{loc}^{1, p}(W)$. Hence, $f \circ u$ is weakly differentiable.

    Let $(u_{k_l})_{l \in \mathbb{N}}$ be a subsequence such that
    \[
    \lim_{l \to \infty} \int_M |du_{k_l}|_{g, h}^p \, d\vol_g = \liminf_{k \to \infty} \int_M |du_k|_{g, h}^p \, d\vol_g,
    \]
    and assume without loss of generality that $|du_{k_l}|_{g, h} \rightharpoonup w$ in $L^p(M)$ for some $w \in L^p(M)$. Let 
    \[
    F \in C_c^1(N; \mathbb{R}^n), \quad v \in C_c^\infty(M; TM), \quad w \in C_c^\infty(M; \mathbb{R}^n).
    \]
    We observe that
    \begin{equation} \label{eq:divergence_trick}
        \int_M (D(F \circ u)(v), w) \, d\vol_g = \sum_{i = 1}^n \int_M D(F_i \circ u)(w_i v) \, d\vol_g = - \sum_{i = 1}^n \int_M (F_i \circ u) \operatorname{div}(w_i v) \, d\vol_g.
    \end{equation}
    Clearly, the equalities hold for $u_k$ as well. Therefore, replacing $u$ by $u_{k_l}$ on the right-hand side of \eqref{eq:divergence_trick} and taking $l \to \infty$, we get
    \begin{equation} \label{eq:weak_convergence_with_vectors}
        \int_M (D(F \circ u)(v), w) \, d\vol_g = \lim_{l \to \infty} \int_M (D(F \circ u_{k_l})(v), w) \, d\vol_g.
    \end{equation}
    If $L \in C_c^\infty(M; T^*M \otimes \mathbb{R}^n)$, then it similarly follows from \eqref{eq:weak_convergence_with_vectors} that
    \begin{equation} \label{eq:weak_convergence_with_transformations}
        \int_M (D(F \circ u), L)_{g, \e_n} \, d\vol_g = \lim_{l \to \infty} \int_M (D(F \circ u_{k_l}), L)_{g, \e_n} \, d\vol_g.
    \end{equation}
    Define the Lipschitz constant $\operatorname{Lip}(F)$ by
    \[
    \operatorname{Lip}(F) := \inf \left\{\frac{|F(x) - F(z)|}{d_g(x, z)} : x, z \in M, x \neq z\right\}.
    \]
    Then $|(D(F \circ u_{k_l}), L)_{g, \e_n}| \le \operatorname{Lip}(F) |du_{k_l}|_{g, h} |L|_{g, \e_n}$, and by the weak convergence of $(|du_{k_l}|_{g, h})$, we have
    \[
    \left|\int_M (D(F \circ u), L)_{g, \e_n} \, d\vol_g\right| \le \operatorname{Lip}(F) \int_M w|L|_{g, \e_n} \, d\vol_g.
    \]
    Since $L$ is arbitrary, we conclude that $|D(F \circ u)|_{g, \e_n} \le \operatorname{Lip}(F) w$ a.e. in $M$. At this point, we refer to Proposition 2.2. in \cite{CVS} to deduce that $|du|_{g, h} \le w$ a.e. in $M$. Therefore,
    \[
    \int_M |du|_{g, h}^p \, d\vol_g \le \int_M w \, d\vol_g \le \liminf_{l \to \infty} \int_M |du_{k_l}|_{g, h}^p \, d\vol_g = \liminf_{k \to \infty} \int_M |du_k|_{g, h}^p \, d\vol_g
    \]
    This concludes the proof.
\end{proof}

\begin{proposition}[{\cite[Proposition 2.6]{CVS}}] \label{pr:lipschitz_composition}
    Let $u : M \rightarrow N$ be a c.w.d. map and let $f : N \rightarrow \mathbb{R}$ be a Lipschitz function. Assume that there exists $q \in N$ such that $d_h(u, q)$ and $|du|_{g, h}$ are locally integrable in $M$. Then $f \circ u$ is weakly differentiable and
    \[
    |d(f \circ u)_x|_{g, \e_d} \le \operatorname{Lip}(f)|du_x|_{g, h} \quad \text{for a.e. } x \in M.
    \]
\end{proposition}

We are now in position to prove a generalization of the Rellich-Kondrachov theorem to homogeneous Sobolev spaces. See Proposition 3.4 in \cite{CVS} for an alternative compactness result.

\begin{theorem} \label{thm:rellich_kondrachov}
    Let $D \subset \mathbb{R}^d$ be open and bounded with Lipschitz boundary and let $N$ be complete. Let $(u_k) \subset \dot{W}^{1, p}(D; N)$ for some $p \in (1, \infty)$. Assume that there exists $q_0 \in N$ and $C > 0$ such that
    \begin{equation} \label{eq:rellich-kondrachov}
        \int_D d_h^p(u_k, q_0)\, dx \le C, \quad \int_D |du_k|_{\e_d, h}^p \, dx \le C.
    \end{equation}
    Then there exist a subsequence $(u_{k_l})$ and $u \in \dot{W}^{1, p}(D; N)$ such that 
    \[
    d_h(u_{k_l}, u) \to 0 \text{ in } L^p(D).
    \]
\end{theorem}

\begin{proof}
    Let $(q_j)$ be a dense sequence in $N$. We define $v_{k, j} := d_h(u_k, q_j)$. By Proposition \ref{pr:lipschitz_composition}, $v_{k, j}$ is weakly differentiable and
    \[
    |v_{k, j}| \le d_h(u_k, q_0) + d_h(q_0, q_j), \quad |Dv_{k, j}| \le |du_k|_{\e_d, h} \quad \text{a.e. in } D.
    \]
    By \eqref{eq:rellich-kondrachov}, $(v_{k, j})_k$ is bounded in $W^{1, p}(D)$ for all $j$. Using the Rellich-Kondrachov theorem and a diagonalization argument, we find functions $v_j \in W^{1, p}(D)$ and a subsequence $(k_l)$ such that
    \[
    v_{k_l, j} \to v_j \in L^p(D), \quad v_{k_l, j} \to v_j \text{ pointwise a.e. in } D
    \]
    for all $j$. Let $Z \subset D$ be a negligible set such that $v_{k_l, j}(x)$ converges for all $x \in D \setminus Z$ and all $j$. We claim $u_{k_l}(x)$ converges in $N$. Let $q$ and $q'$ be limit points of $(u_{k_l}(x))$. If $q \neq q'$, then we can find $j$ such that $d_h(q, q_j) < 2d_h(q', q_j)$. In this case, $(v_{k_l, j}(x))$ clearly cannot converge. Hence, $q = q'$, and $(u_{k_l}(x))$ is convergent.
    
    We set $u(x) := \lim_{l \to \infty} u_{k_l}(x)$ for $x \in D \setminus Z$ and $u(x) = q_0$ for $x \in Z$. By Lemma \ref{lm:closure_under_pointwise_convergence}, we conclude that $u \in \dot{W}^{1, p}(D; N)$. Furthermore, Fatou's lemma gives
    \begin{equation} \label{eq:application_fatou's_lemma}
        \int_D d_h^p(u, q_0) \, dx \le \liminf_{l \to \infty} \int_D d_h^p(u_{k_l}, q_0) \, dx.    
    \end{equation}
    To finish the proof, we need to show that $d_h(u_{k_l}, u) \to 0$ in $L^p(D)$. We start by noting that $d_h(u_{k_l}, q)$ converges in $L^p(D)$ for all $q \in N$, since $(v_{k_l, j})$ is convergent in $L^p(D)$ and $(q_j)$ is dense in $N$. As a consequence, $d_h(u_{k_l}, v)$ converges for all simple functions $v : D \rightarrow N$ and, by approximation, for all $v$ such that $d_h(v, q_0) \in L^p(D)$. Hence, $d_h(u_{k_l}, u)$ converges in $L^p(D)$ by \eqref{eq:application_fatou's_lemma}. However, since the pointwise limit is $0$ a.e. in $D$, the $L^p$ limit has to be $0$ as well.
\end{proof}

We end this section with a variation of Poincaré's inequality on cubes for homogeneous Sobolev maps.

\begin{lemma} \label{lm:poincare_manifold}
    Let $1 < p < \infty$. Let $Q \subset \mathbb{R}^d$ be an open and bounded cube and $u \in \dot{W}^{1, p}(Q; N)$. Assume that
    \begin{equation} \label{eq:lp_integrable}
        \int_Q d_h^p(u, q_0) \, dx < \infty
    \end{equation}
    for some $q_0 \in N$. Then
    \begin{equation} \label{eq:distance_poincaré}
        \int_Q \int_Q d_h^p(u(x), u(z)) \, dx \, dz \le C\diam^p(Q) |Q| \int_Q |du|_{\e_d, h}^p \, dx,
    \end{equation}
    where $C$ depends only on $d$ and $p$.
\end{lemma}

\begin{proof}
    We start by proving the inequality in $B(0, 1)$. Let $v \in W^{1, p}(B(0, 1))$. By Poincaré's inequality on balls (see \cite[Lemma 4.1]{EG}), we have
    \begin{equation} \label{eq:poincaré_on_balls}
        \int_{B(0, 1)} |v(x) - v(z)|^p \, dx \le C \int_{B(0, 1)} |Dv(x)|^p |x - z|^{1 - n} \, dx \quad \text{for a.e. } z \in B(0, 1).
    \end{equation}
    Let $(q_j)$ be a dense sequence in $N$ and set $v_j := d_h(u, q_j)$. By Proposition \ref{pr:lipschitz_composition}, $v_j$ is weakly differentiable and 
    \[
    |v_j| \le d_h(u, q_0) + d_h(q_0, q_j), \quad |Dv_j| \le |du|_{\e_d, h} \quad \text{a.e. in } B(0, 1),
    \]
    so that $v_j \in W^{1, p}(B(0, 1))$. Therefore, \eqref{eq:poincaré_on_balls} gives a negligible set $Z$ such that
    \[
    \int_{B(0, 1)} |v_j(x) - v_j(z)|^p \, dx \le C \int_{B(0, 1)} |du_x|_{\e_d, h}^p |x - z|^{1 - n} \, dx
    \]
    for all $z \in B(0, 1) \setminus Z$ and $j \in \mathbb{N}$. Fix $z \in B(0, 1) \setminus Z$ and choose a subsequence $(q_{j_i})$ converging to $u(z)$. Since $|v_j(x) - v_j(z)| \le d_h(u(x), u(z))$, the bound \eqref{eq:lp_integrable} and the dominated convergence theorem imply
    \[
    \int_{B(0, 1)} d_h(u(x), u(z))^p \, dx = \lim_{i \to \infty} \int_{B(0, 1)} |v_{j_i}(x) - v_{j_i}(z)|^p \, dx \le C \int_{B(0, 1)} |du_x|_{\e_d, h}^p |x - z|^{1 - n} \, dx. 
    \]
    Finally, integrating both sides in $z$, we obtain
    \[
    \int_{B(0, 1)} \int_{B(0, 1)} d_h^p(u(x), u(z)) \, dx \, dz \le C \int_{B(0, 1)} |du_x|_{\e_d, h}^p \, dx,
    \]
    Since the unit cube and the unit ball can be mapped to each other by a bilipschitz map, the inequality on the unit cube follows by a change of variables. We conclude the proof by noting that \eqref{eq:distance_poincaré} is invariant under scaling.
\end{proof}

\subsubsection{Sobolev immersions} \label{sec:sobolev_immersions}

In this section we introduce the main function space we will work with. The following definitions are based on \cite[Sections 2.4, 2.5]{AKM}. Let $(M, g)$ and $(N, h)$ be oriented Riemannian manifolds of $d$ and $d + 1$ dimensions, respectively. Let $1 < p < \infty$. If $u \in \dot{W}^{1, p}(M; N)$ and $\rank du_x = d$, then there exists a unique vector $\nu_u(x) \in T_{u(x)}N$ satisfying the following conditions:
\begin{enumerate}
    \item $|\nu_u(x)|_h = 1$,
    \item $(\nu_u(x), du_x(v))_h = 0$ for all $v \in T_xM$,
    \item if $(v_1, \dots, v_d)$ is a positively oriented basis of $T_xM$, then $(du_x(v_1), \dots, du_x(v_d), \nu_u(x))$ is a positively oriented basis of $T_{u(x)}N$.
\end{enumerate}
We call $\nu_u(x)$ the oriented unit normal at $x$. If $\rank du_x < d$, we set $\nu_u(x) := 0 \in T_{u(x)}N$. We define the space of Sobolev immersions by
\[
\Imm_p(M; N) := \{u \in \dot{W}^{1, p}(M; N) : \rank du_x = d \text{ for a.e. } x \in M, \, \nu_u \in \dot{W}^{1, p}(M; TN)\}.
\]
Given $u \in \Imm_p(M; N)$, we can consider it as a deformation of $M$. We introduce the following functionals that quantify its stretching and bending energies:
\begin{align}
    E_s(u) &:= \int_M \dist_{g, h}^p(du_x, \Ort((T_xM, g_x), (T_{u(x)}N, h_{u(x)}))) \, d\vol_g(x), \\
    E_b(u) &:= \int_M |K_{TN} \circ d(\nu_u)_x|_{g, h}^p \, d\vol_g(x).    
\end{align}
We recall that if $u$ is smooth, $K_{TN} \circ d(\nu_u)$ gives the covariant derivative of $\nu_u$. Since $du_x$ is injective for a.e. $x \in M$, there exists a unique linear map $S_u(x) : T_xM \rightarrow T_xM$ satisfying
\[
du_x \circ S_u(x) = K_{TN} \circ d(\nu_u)_x
\]
for a.e. $x \in M$. We call $S_u : TM \rightarrow TM$ the induced shape operator. If $M$ has a pre-assigned ``shape", then we can also measure the deviation from the original shape after the deformation. Assume that $M$ is equipped with a smooth symmetric 2-tensor field $b$. For each $x \in M$ we let $S(x) : T_xM \rightarrow T_xM$ be the unique linear map defined by
\begin{equation} \label{eq:reference_shape_operator}
    (S(x)v, w)_{g_x} = b(v, w) \quad \text{for all } v, w \in T_xM.
\end{equation}
We call the resulting map $S : TM \rightarrow TM$ the reference shape operator. The modified bending energy is given by
\begin{equation}
    E_b^S(u) := \int_M |du_x \circ (S_u(x) - S(x))|_{g_x, h_{u(x)}}^p \, d\vol_g(x).
\end{equation}
Clearly, $E_b(u) = E_b^S(u)$ if $S \equiv 0$.

\section{The quantitative rigidity estimate on \texorpdfstring{$\varepsilon$}{epsilon}-isometric charts} \label{sec:quantitative}

Let $d \in \mathbb{N}$ and let $(N, h)$ be an oriented $(d + 1)$-dimensional Riemannian manifold. In this section, we prove a local quantitative rigidity result for Sobolev immersions defined on a subset of $\mathbb{R}^d$ and taking values in $N$. This is achieved in Theorem \ref{thm:local_quantitative_rigidity}. Our strategy will be to put coordinates on $N$ to work with vector valued functions to which we can apply rigidity results in Euclidean spaces. In order to preserve as much metric information as possible in the transition to coordinates, we need to work with close-to-isometric charts, which we introduce in the next definition.

\begin{definition} \label{def:almost_isometric_chart}
    Let $(U, \varphi)$ be chart on $N$ and $\varepsilon \in (0, 1)$. Denote the pushforward of the metric $h$ by $\varphi$ by the same letter and let $\Gamma_{i, j}^k : U \rightarrow \mathbb{R}$ for $i, j , k = 1, \dots, d + 1$ be the Christoffel symbols for $h$. The chart $(U, \varphi)$ is called $\varepsilon$-isometric if
    \begin{equation} \label{eq:epsilon_isometric_chart}
        \frac{1}{1 + \varepsilon} \e_{d + 1} < h < (1 + \varepsilon) \e_{d + 1} \text{ in } \varphi(U), \quad \|\Gamma_{i, j}^k\|_{L^\infty(U)} \le \varepsilon \text{ for all } i, j, k \in \{1, \dots, d + 1\}.
    \end{equation}
\end{definition}

For all $q \in N$ and $\varepsilon > 0$, the restriction of the exponential map on $T_qN$ to a sufficiently small ball centered at $0$ gives an $\varepsilon$-isometric chart. In the next lemma, we bound quantities that appear in Euclidean rigidity estimates in terms of the corresponding quantities on the manifold.

\begin{lemma} \label{lm:non_euclidean_to_euclidean}
    Let $(U, \varphi)$ be an $\varepsilon$-isometric chart on $N$. Let $q \in U$ and set $y := \varphi(q)$.
    \begin{enumerate}
        \item Let $g_0$ be a constant metric on $\mathbb{R}^d$. If $L \in \mathcal{L}(\mathbb{R}^d, T_qN)$, then
        \begin{equation} \label{eq:distance_from_o_isometric_chart}
            \dist_{g_0, \e_{d + 1}}(D\varphi(q) \circ L, \Ort(g_0, \e_{d + 1}) \le \sqrt{1 + \varepsilon} \dist_{g_0, h_q}(L, \Ort((\mathbb{R}^d, g_0), (T_qN, h_q)) + C\varepsilon.
        \end{equation}
        \item Let $\Pi \subset \mathbb{R}^{d + 1}$ be a $d$-dimensional subspace. Let $n \in \mathbb{R}^{d + 1}$ be normal to $\Pi$ and let $\nu \in T_qN$ be normal to $d(\varphi^{-1})_y(\Pi)$. Assume that $|n| = 1$, $|\nu|_h = 1$ and assume that $n$ and $\tilde{n} := D\varphi(q)(\nu)$ induce the same orientation on $\Pi$. Then $|n - \tilde{n}| \le C\varepsilon$.
    \end{enumerate}
    Both constants $C$ depend only on $d$.
\end{lemma}

\begin{proof}
    We denote by $h_y$ the pushforward of the metric $h_q$ by $\varphi$. Define the matrix $H_y$ by 
    \[
    (H_y)_{ij} := \left(\frac{\partial}{\partial y_i}\bigg|_y, \frac{\partial}{\partial y_j}\bigg|_y\right)_{h_y}, \quad i, j = 1, \dots, d + 1.
    \]
    Let $(v_1, \dots, v_{d + 1})$ be an orthonormal basis of $\mathbb{R}^{d + 1}$ consisting of eigenvectors of $H_y$. Set $\mu_i := (v_i, v_i)_{h_y}$ and $\tilde{v}_i := v_i/\sqrt{\mu_i}$. By \eqref{eq:epsilon_isometric_chart}, we have $(1 + \varepsilon)^{-1} < \mu_i < 1 + \varepsilon$. Define $S \in \mathcal{L}(\mathbb{R}^{d + 1})$ by $S\tilde{v}_i := v_i$, so that $S \in \Ort(h_y, e_{d + 1})$. Then
    \begin{equation} \label{eq:comparison_with_identity}
        |\Id_{\mathbb{R}^{d + 1}} - S|_{h_y, \e_{d + 1}}^2 = \sum_{i = 1}^{d + 1} |\tilde{v}_i - S\tilde{v_i}|^2 = \sum_{i = 1}^{d + 1} \left|\frac{1}{\sqrt{\mu_i}} - 1\right|^2 \le 4 \sum_{i = 1}^{d + 1} |1 - \mu_i|^2 = 4(d + 1)\varepsilon^2.
    \end{equation}
    Now, we choose $R \in \Ort((\mathbb{R}^d, g_0), (T_qN, h_q))$ such that
    \[
    |L - R|_{g_0, h_q} = \dist_{g_0, h_q}(L, \Ort((\mathbb{R}^d, g_0), (T_qN, h_q)).
    \]
    Then $S \circ D\varphi(q) \circ R \in \Ort(g_0, \e_{d + 1})$ and
    \begin{align*}
        \dist&_{g_0, \e_{d + 1}}(D\varphi(q) \circ L, \Ort(g_0, \e_{d + 1})) \\
        &\le |D\varphi(q) \circ L - S \circ D\varphi(q) \circ R|_{g_0, \e_{d + 1}} \le |D\varphi(q) \circ (L - R)|_{g_0, \e_{d + 1}} + |(\Id_{\mathbb{R}^{d + 1}} - S) \circ D\varphi(q) \circ R|_{g_0, \e_{d + 1}} \\
        &\le \sqrt{1 + \varepsilon} |D\varphi(q) \circ (L - R)|_{g_0, h_y} + |\Id_{\mathbb{R}^{d + 1}} - S|_{h_y, \e_{d + 1}} \le \sqrt{1 + \varepsilon} |L - R|_{g_0, h_q} + C\varepsilon.
    \end{align*}
    Hence,
    \[
    \dist_{g_0, \e_{d + 1}}(D\varphi(q) \circ L, \Ort(g_0, \e_{d + 1})) \le \sqrt{1 + \varepsilon} \dist_{g_0, h_q}(L, \Ort((\mathbb{R}^d, g_0), (T_qN, h_q)) + C\varepsilon.
    \]
    We move on to the second claim. If $v, \tilde{v} \in \mathbb{R}^{d + 1}$ and
    \[
    v = \sum_{i = 1}^{d + 1} a_i v_i, \quad \tilde{v} = \sum_{i = 1}^{d + 1} b_i v_i,
    \]
    then
    \begin{equation} \label{eq:inner_products_difference}
        |(v, \tilde{v}) - (v, \tilde{v})_{h_y}| = \left|\sum_{i = 1}^{d + 1} a_i b_i (1 - \mu_i)\right| \le \varepsilon |v| |\tilde{v}|.
    \end{equation}
    An orthogonal decomposition of $\tilde{n}$ yields a unit vector $w \in \Pi$ such that $\tilde{n} = (\tilde{n}, n)n + (\tilde{n}, w)w$. By \eqref{eq:inner_products_difference}, we get
    \[
    |(\tilde{n}, \tilde{n}) - 1| = |(\tilde{n}, \tilde{n}) - (\tilde{n}, \tilde{n})_{h_y}| \le \varepsilon |\tilde{n}|^2 \le 4\varepsilon, \quad |(\tilde{n}, w)| = |(\tilde{n}, w) - (\tilde{n}, w)_{h_y}| \le \varepsilon |\tilde{n}| |w| \le 2\varepsilon.
    \]
    Therefore, $|(\tilde{n}, n)| > |\tilde{n}| - |(\tilde{n}, w)| \ge 1 - 6\varepsilon$. However, $(\tilde{n}, n) > 0$, since $\tilde{n}$ and $n$ induce the same orientation on $\Pi$. Thus, we conclude that $|\tilde{n} - n| \le C\varepsilon$.
\end{proof}

The composition of a Sobolev immersion and a coordinate chart is not weakly differentiable in general, since its domain is not necessarily open. In the next definition, we extend the domain of a chart to all of $N$ by a cut-off procedure. This allows us to obtain compactly supported coordinate maps that are diffeomorphic on an open subset of $N$. 

\begin{definition} \label{def:extension}
    Let $(U, \varphi)$ be a chart on $N$ centered at $q$, that is, $\varphi(q) = 0$. Let $\theta \in C_c^\infty(\mathbb{R}^{d + 1}; \mathbb{R}^{d + 1})$ and assume that $\theta = \Id_{\mathbb{R}^{d + 1}}$ in $B(0, 1)$ and $\theta \equiv 0$ in $\mathbb{R}^{d + 1} \setminus B(0, 2)$. If $B(0, 2r) \subset \varphi(U)$, then we set $\varphi^{(r)}(x) := r\theta(\varphi(x)/r)$ for $x \in U$ and $\varphi^{(r)}(x) := 0$ for $x \in N \setminus U$. We say that $\varphi^{(r)}$ extends $(U, \varphi)$ by $\theta$.
\end{definition}

If $\varphi^{(r)}$ extends $(U, \varphi)$ by $\theta$, then we note that $\supp(\varphi^{(r)}) \subset U$ and $\varphi^{(r)} = \varphi$ in $\varphi^{-1}(B(0, r))$. If $u$ is a Sobolev immersion taking values in $N$, then $\varphi^{(r)} \circ u$ and $D\varphi^{(r)} \circ \nu_u$ are almost isometric coordinate representations of $u$ and the unit normal vector field $\nu_u$, respectively, for $x \in u^{-1}(\varphi^{-1}(B(0, r)))$. Our next goal is to bound the derivative of $D\varphi^{(r)} \circ \nu_u$ in terms of the covariant derivative $K_{TN} \circ d\nu_u$. We begin with a lemma.

\begin{lemma} \label{lm:game_of_connectors}
    Let $(U, \varphi)$ be an $\varepsilon$-isometric chart on $N$ and let $\xi \in T_{(q, v)}TN$ with $q \in U$. Then
    \[
    |K_{T\mathbb{R}^{d + 1}}(d^2\varphi_{(q, v)}(\xi))| \le C(|K_{TN}(\xi)|_h + |v|_h |d(\pi_N)_{(q, v)}(\xi)|_h),
    \]
    where the constant $C$ depends only on $d$.
\end{lemma}

\begin{proof}
    Denote the local coordinates on $U$ by $(y_1, \dots, y_{d + 1})$. We define coordinates on $TU$ by
    \[
    (y, w) \in U \times \mathbb{R}^{d + 1}\mapsto \sum_{i = 1}^{d + 1} w_i \frac{\partial}{\partial y_i}\bigg|_{\varphi^{-1}(y)} \in TU.
    \]
    Then
    \[
    \xi = \sum_{i = 1}^{d + 1} Y_i \frac{\partial}{\partial y_i}\bigg|_{(q, v)} + \sum_{i = 1}^{d + 1} W_i \frac{\partial}{\partial w_i}\bigg|_{(q, v)}
    \]
    for some $Y, W \in \mathbb{R}^{d + 1}$. By \eqref{eq:connector_operator_coordinates} and \eqref{eq:connector_operator_euclidean}, we have
    \begin{gather*}
        d\varphi_q(K_{TN}(\xi)) = \sum_{k = 1}^{d + 1} \left(W_k + \sum_{i, j = 1}^{d + 1} w_j Y_i \Gamma_{i, j}^k(q)\right) \frac{\partial}{\partial y_k}\bigg|_y, \\
        K_{T\mathbb{R}^{d + 1}}(d^2\varphi_{(q, v)}(\xi)) = \sum_{k = 1}^{d + 1} W_k \frac{\partial}{\partial y_k}\bigg|_y,
    \end{gather*}
    where $(y, w)$ are the coordinates of $(q, v)$. We know that
    \[
    d\varphi_q(v) = \sum_{i = 1}^{d + 1} w_i \frac{\partial}{\partial y_i}\bigg|_y, \quad d(\varphi \circ \pi_N)_{(q, v)}(\xi) = \sum_{i = 1}^{d + 1} Y_i \frac{\partial}{\partial y_i}\bigg|_y.
    \]
    Consequently,
    \[
    |d\varphi_q(K_{TN}(\xi)) - K_{T\mathbb{R}^{d + 1}}(d^2\varphi_{(q, v)}(\xi))| \le |d\varphi_q(v)| |d(\varphi \circ \pi_N)_{(q, v)}(\xi)| \le 4 |v|_h |d(\pi_N)_{(q, v)}(\xi)|_h.
    \]
    Since $|d\varphi_q(K_{TN}(\xi))| \le 2|K_{TN}(\xi)|_h$, we are done.
    
\end{proof}

\begin{lemma} \label{lm:euclidean_normal_bound}
    Let $(U, \varphi)$ be an $\varepsilon$-isometric chart on $N$ and let $\varphi^{(r)}$ be its extension by $\theta$. Let $Q \subset \mathbb{R}^d$ be an open and bounded cube and let $u \in \Imm_p(Q; N)$. Define $\tilde{n}(x)$ by
    \begin{equation} \label{eq:euclidean_normal_definition}
        d\varphi^{(r)}_{u(x)}(\nu_u(x)) = \sum_{i = 1}^{d + 1} \tilde{n}_i(x) \frac{\partial}{\partial y_i}\bigg|_{\varphi^{(r)}(u(x))}, \quad x \in Q.
    \end{equation}
    In other words, $\tilde{n}(x) = D\varphi^{(r)}(u(x))(\nu_u(x))$. Set $F := u^{-1}(\varphi^{-1}(B(0, r)))$. Then $\tilde{n} \in W^{1, p}(Q; \mathbb{R}^{d + 1})$ and 
    \[
    \int_Q |D\tilde{n}|^p \, dx \le C\left(\frac{1}{r^p}\int_{Q \setminus F} |du|_{\e_d, h}^p \, dx + \int_Q |du|_{\e_d, h}^p \, dx + E_b(u)\right).
    \]
    The constant $C$ depends only on $d$, $p$ and $\theta$.
\end{lemma}

\begin{proof}
    We start by showing that $\tilde{n}$ is a c.w.d. map. Let $O$ be a bounded and open set such that $\supp(\varphi^{(r)}) \subset O \csubset U$. We define the compact set
    \[
    G := \{(q, v) \in TN : q \in \bar{O}, |v|_h \le 2\}.
    \]
    We choose $\Phi \in C_c^\infty(TN)$ such that $\Phi \equiv 1$ in $G$. Since $|\nu_u|_h = 1$ a.e. in $Q$, we have
    \begin{equation} \label{eq:euclidean_normal}
        \tilde{n} = D\varphi^{(r)} \circ \nu_u = (D\varphi^{(r)} \circ (\Phi \cdot \Id_{TN})) \circ \nu_u.
    \end{equation}
    Clearly, $D\varphi^{(r)} \circ (\Phi \cdot \Id_{TN}) \in C_c^\infty(TN; \mathbb{R}^{d + 1})$. Since $\nu_u$ is a c.w.d., this implies that $\tilde{n}$ is a c.w.d. map. Consequently, $\tilde{n}$ is c.w.d. as well. However, $\tilde{n} \in L^\infty(Q)$, so that it is actually weakly differentiable.

    Next, we compute $d\tilde{n}$. It follows from \eqref{eq:euclidean_normal_definition} and the second claim in Theorem \ref{th:colocal_derivative_existence} that
    \begin{equation} \label{eq:euclidean_normal_representation}
        d\tilde{n} = K_{T\mathbb{R}^{d + 1}} \circ d(d\varphi^{(r)} \circ \nu_u) = K_{T\mathbb{R}^{d + 1}} \circ d^2\varphi^{(r)} \circ d\nu_u \quad \text{a.e. in } Q.
    \end{equation}
    We define $\theta^{(r)}(y) := r\theta(y/r)$. Then $\varphi^{(r)} = \theta^{(r)} \circ \varphi$ in $U$. If we identify $TT\mathbb{R}^{d + 1}$ with $\mathbb{R}^{4(d + 1)}$ as in \eqref{eq:double_tangent_bundle}, then a routine computation yields for all $(y, w, Y, W) \in TT\mathbb{R}^{d + 1}$,
    \begin{equation} \label{eq:connector_of_second_differential}
        \begin{aligned}
            (K_{T\mathbb{R}^{d + 1}} \circ d^2\theta^{(r)})(y, w, Y, W) &= \sum_{i = 1}^{d + 1} \left(D\theta_i^{(r)}(y)(W) + (H\theta_i^{(r)}(y)w, Y)\right) \frac{\partial}{\partial y_i}\bigg|_y \\
            &= \sum_{i = 1}^{d + 1} \left(D\theta_i\left(\frac{y}{r}\right)(W) + \frac{1}{r}\left(H\theta_i\left(\frac{y}{r}\right)w, Y\right)\right) \frac{\partial}{\partial y_i}\bigg|_y,
        \end{aligned}
    \end{equation}
    where $H$ stands for the Hessian matrix. For $(x, v) \in T\mathbb{R}^{d + 1}$ with $x \in u^{-1}(U)$, we define $\tilde{y}(x)$, $\tilde{w}(x)$, $\tilde{Y}(x, v)$ and $\tilde{W}(x, v)$ by
    \begin{gather*}
        \tilde{y}(x) := \varphi(u(x)), \quad \tilde{w}(x) := d\varphi_{u(x)}(\nu_u(x)),\\
        d^2\varphi_{\nu_u(x)}(d(\nu_u)_x(v)) = \sum_{i = 1}^{d + 1} \tilde{Y}_i(x, v) \frac{\partial}{\partial y_i}\bigg|_{(\tilde{y}(x), \tilde{w}(x))} + \sum_{i = 1}^{d + 1} \tilde{W}_i(x, v) \frac{\partial}{\partial w_i}\bigg|_{(\tilde{y}(x), \tilde{w}(x))}.
    \end{gather*}
    Hence, \eqref{eq:euclidean_normal_representation} and \eqref{eq:connector_of_second_differential} imply
    \begin{equation} \label{eq:computation_of_differential_euclidean_normal}
        d\tilde{n}_x = \sum_{k = 1}^{d + 1} \left(D\theta_k\left(\frac{\tilde{y}(x)}{r}\right)(\tilde{W}(x, \cdot)) + \frac{1}{r}\left(H\theta_k\left(\frac{\tilde{y}(x)}{r}\right)\tilde{w}(x), \tilde{Y}(x, \cdot)\right)\right) \frac{\partial}{\partial y_k}\bigg|_{\tilde{y}(x)}
    \end{equation}
    for a.e. $x \in u^{-1}(U)$. We proceed to estimate $|d\tilde{n}_x|$. First of all, we note that
    \[
    \tilde{Y} = d\pi_{\mathbb{R}^{d + 1}} \circ d^2\varphi \circ d\nu_u, \quad \tilde{W} = K_{T\mathbb{R}^{d + 1}} \circ d^2\varphi \circ d\nu_u.
    \]
    We can simplify the expression for $\tilde{Y}$ by observing that $u = \pi_N \circ \nu_u$ and $\pi_{\mathbb{R}^{d + 1}} \circ d\varphi = \varphi \circ \pi_N$, so that
    \[
    d\pi_{\mathbb{R}^{d + 1}} \circ d^2\varphi \circ d\nu_u = d(\pi_{\mathbb{R}^{d + 1}} \circ d\varphi) \circ d\nu_u = d(\varphi \circ \pi_N) \circ d\nu_u = d\varphi \circ du \quad \text{a.e. in } u^{-1}(U).
    \]
    Therefore
    \begin{equation} \label{eq:explicit_representation}
        \tilde{Y}(x, \cdot) = d\varphi_{u(x)} \circ du_x \quad \text{for a.e. } x \in u^{-1}(U).
    \end{equation}
    On the other hand, we can use Lemma \ref{lm:game_of_connectors} to obtain
    \[
    |K_{T\mathbb{R}^{d + 1}} \circ d^2\varphi_{\nu_u(x)} \circ d(\nu_u)_x| \le C(|K_{TN} \circ d(\nu_u)_x|_{\e_d, h} + |\nu_u(x)|_h |d(\pi_N)_{\nu_u(x)} \circ d(\nu_u)_x|_{\e_d, h})
    \]
    for $x \in u^{-1}(U)$. Since $|\nu_u(x)|_h = 1$ a.e. in $Q$ and $\pi_N \circ \nu_u = u$, we have
    \begin{equation} \label{eq:explicit_estimate}
        |\tilde{W}(x, \cdot)| \le C(|K_{TN} \circ d(\nu_u)_x|_{\e_d, h} + |du_x|_{\e_d, h}) \quad \text{for a.e. } x \in u^{-1}(U).
    \end{equation}
    We can now estimate $|d\tilde{n}_x|$ for a.e. $x \in F = u^{-1}(\varphi^{-1}(B(0, r))$. Recall that $\theta \equiv \Id_{\mathbb{R}^{d + 1}}$ in $B(0, 1)$. Hence, by \eqref{eq:computation_of_differential_euclidean_normal},
    \[
    d\tilde{n}_x = \sum_{k = 1}^{d + 1} \tilde{W}_k(x, \cdot) \frac{\partial}{\partial y_k}\bigg|_{\tilde{y}(x)} \quad \text{for a.e. } x \in F.
    \]
    Therefore, \eqref{eq:explicit_estimate} yields
    \begin{equation} \label{eq:close_to_isometric_derivative_estimate}
        |d\tilde{n}_x| \le C(|K_{TN} \circ d(\nu_u)_x|_{\e_d, h} + |du_x|_{\e_d, h}) \quad \text{for a.e. } x \in F.
    \end{equation}
    To give a bound in $u^{-1}(U) \setminus F$, we simply apply the triangle inequality to \eqref{eq:computation_of_differential_euclidean_normal} and use \eqref{eq:explicit_representation} and \eqref{eq:explicit_estimate}:
    \begin{align*}
        |d\tilde{n}_x| &\le \|D\theta\|_{L^\infty(\mathbb{R}^{d + 1})}|\tilde{W}(x, \cdot)| + \frac{C}{r}|\tilde{w}(x)||\tilde{Y}(x, \cdot)|\sum_{k = 1}^{d + 1} \|H\theta_k\|_{L^\infty(\mathbb{R}^{d + 1})} \\
        &\le C\left(|K_{TN} \circ d(\nu_u)_x|_{\e_d, h} + |du_x|_{\e_d, h}\right) + \frac{C}{r}|d\varphi_{u(x)}(\nu_u(x))||d\varphi_{u(x)} \circ du_x| \\
        &\le C|K_{TN} \circ d(\nu_u)_x|_{\e_d, h} + \frac{C}{r}|du_x|_{\e_d, h} \quad \text{for a.e. } x \in u^{-1}(U) \setminus F,
    \end{align*}
    where the last constant depends on $\theta$. Thus, to conclude, the previous estimate together with \eqref{eq:close_to_isometric_derivative_estimate} gives the desired inequality:
    \[
    \int_Q |d\tilde{n}|^p \, dx \le C\left(\frac{1}{r^p}\int_{Q \setminus F} |du|_{\e_d, h}^p \, dx + \int_Q |K_{TN} \circ d\nu_u|_{e_d, h}^p + |du|_{\e_d, h}^p \, dx\right).
    \]
    This also proves that $\tilde{n} \in W^{1, p}(Q; \mathbb{R}^{d + 1})$.
\end{proof}

For the remainder of this section, we fix an open and bounded cube $Q \subset \mathbb{R}^d$ endowed with a metric $g$ such that
\[
\frac{1}{\lambda} \e_d \le g \le \lambda \e_d \quad \text{in } Q
\]
for some $\lambda > 0$. In the proof of our local quantitative rigidity result for Sobolev immersions, we rely on a generalization of the Friesecke-James-Müller rigidity estimate (\cite[Theorem 3.1]{FJM}, see also \cite[Section 2.4]{CS}) to domains endowed with a non-Euclidean metric. Compare the following result with \cite[Theorem 2.3]{LP}.

\begin{theorem} \label{thm:noneuclidean_rigidity}
    For every $u \in W^{1, p}(Q; \mathbb{R}^d)$ and $x_0 \in Q$, there exists $R \in \SO(g_{x_0}, \e_d)$ such that
    \[
    \int_Q |Du - R|_{g, \e_d}^p \, dx \le C\left(|Q|\left(\osc_Q g\right)^p + \int_Q \dist_{g, \e_d}^p(Du, \SO(g, \e_d)) \, dx\right).
    \]
    The constant $C$ depends only on $p$, $d$ and $\lambda$.
\end{theorem}

A proof of the theorem above is given in \cite{B}. Finally, we state and prove the local quantitative rigidity result for Sobolev immersions. 

\begin{theorem} \label{thm:local_quantitative_rigidity}
    Let $(U, \varphi)$ be an $\varepsilon$-isometric chart on $N$ and let $\varphi^{(r)}$ be its extension by $\theta$. Let $u \in \Imm_p(Q; N)$. Let $\delta \in (0, 1)$ and assume that $F \subset Q$ is a Lebesgue measurable set such that
    \[
    u(F) \subset \varphi^{-1}(B(0, r)), \quad \frac{|Q \setminus F|}{|Q|} \le \delta^p.
    \]
    Then there exists $x_0 \in Q$ and $R \in \Ort(g_x, \e_{d + 1})$ such that
    \[
    \int_Q |D(\varphi^{(r)} \circ u) - R|_{g, \e_{d + 1}}^p \, dx \le C\left(|Q|\left(\delta^p + \varepsilon^p + \left(\osc_Q g\right)^p\right) + E_s(u) + \frac{\diam^p(Q)}{1 - \delta^p} \mathcal{E}(u)\right),
    \]
    where
    \[
    \mathcal{E}(u) := \frac{1}{r^p}\int_{Q \setminus F} |du|_{\e_d, h}^p \, dx + \int_Q |du|_{\e_d, h}^p \, dx + E_b(u).
    \]
    The constant $C$ depends only on $d$, $p$, $\lambda$ and $\theta$.
\end{theorem}

\begin{proof}
    Set $\tilde{u} := \varphi^{(r)} \circ u$. Then $\tilde{u}$ is weakly differentiable, since $u$ is a c.w.d. map. In fact, $\tilde{u} \in W^{1, p}(Q; \mathbb{R}^{d + 1})$. Moreover, $\rank D\tilde{u} = d$ a.e. in $F$ because the restriction of $\varphi^{(r)}$ to $\varphi^{-1}(B(0, r))$ is a diffeomorphism. Thus, for a.e. $x \in F$, there exists a unit vector $n(x) \in \mathbb{R}^{d + 1}$ such that $n(x)$ is orthogonal to $D\tilde{u}(x)(\mathbb{R}^d)$ and $(\partial_{x_1}\tilde{u}(x), \dots, \partial_{x_d} \tilde{u}(x), n(x))$ is positively oriented. We also define $\tilde{n}$ to be the pushforward of $\nu_u$ under $\varphi^{(r)}$ as in \eqref{eq:euclidean_normal_definition}. Our strategy to prove the estimate is to project $\tilde{u}$ onto a $d$-dimensional subspace of $\mathbb{R}^{d + 1}$ and to apply the non-Euclidean rigidity estimate in Theorem \ref{thm:noneuclidean_rigidity}. We choose a $d$-dimensional subspace such that the error in projection is bounded by the $L^p$-norm of $D\tilde{n}$. Since $u$ is represented almost isometrically by $\tilde{u}$ in $F$ and $F$ has a relatively large measure in $Q$, the final estimate will be sharp up to some truncation error in $Q \setminus F$.
    
    By Lemma \ref{lm:euclidean_normal_bound}, $\tilde{n} \in W^{1, p}(Q; \mathbb{R}^{d + 1})$, so that we can apply Poincaré's inequality to get
    \[
    \int_Q \int_Q |\tilde{n}(x) - \tilde{n}(z)|^p \, dx \, dz \le C \diam^p(Q) |Q| \int_Q |D\tilde{n}(x)|^p \, dx.
    \]
    Treating the inner integral on the left-hand side as a function of $z$ and applying Chebyshev's inequality to it, we see that the set of all $z \in Q$ satisfying
    \begin{equation} \label{eq:projection_point}
        \int_Q |\tilde{n}(x) - \tilde{n}(z)|^p \, dx \le \frac{2C}{1 - \delta^p} \diam^p(Q) \int_Q |D\tilde{n}(x)|^p \, dx.
    \end{equation}
    is at least $|Q|(1 - \delta^p)/2$. We observe that for a.e. $x \in F$, $n(x)$ and $\nu_u(x)$ satisfy the hypotheses of part 2 of Lemma \ref{lm:non_euclidean_to_euclidean} with $\Pi = D\tilde{u}(x)(\mathbb{R}^d)$. Hence, 
    \begin{equation} \label{eq:difference_in_normals}
        |n(x) - \tilde{n}(x)| \le C\varepsilon \quad \text{for a.e. } x \in F.
    \end{equation}
    Therefore, we can find $x_0 \in F$ at which both \eqref{eq:projection_point} and \eqref{eq:difference_in_normals} are satisfed. Consequently,
    \begin{align*}
        \int_F |n(x) - n(x_0)|^p \, dx \le C\varepsilon^p|F| + C\int_F |\tilde{n}(x) - &\tilde{n}(x_0)|^p \, dx \\
        &\le C\varepsilon^p|Q| + C\frac{\diam^p(Q)}{1 - \delta^p} \int_Q |D\tilde{n}(x)|^p \, dx.
    \end{align*}
    Then by Lemma \ref{lm:euclidean_normal_bound}, we obtain
    \begin{equation} \label{eq:euclidean_normal_variation}
        \int_F |n(x) - n(x_0)|^p \, dx \le C\varepsilon^p|Q| + C\frac{\diam^p(Q)}{1 - \delta^p} \mathcal{E}(u).
    \end{equation}
    We denote $\lspan(n(x_0))^\perp$ briefly by $\Pi_0$. Let $P$ be the orthogonal projection from $\mathbb{R}^{d + 1}$ onto $\Pi_0$. Let $T \in \SO((\mathbb{R}^d, \e_d), (\Pi_0, \e_{d + 1}))$ and set $v := T^{-1} \circ P \circ \tilde{u}$. Then $v \in W^{1, p}(Q; \mathbb{R}^d)$, and we can apply Theorem \ref{thm:noneuclidean_rigidity} to obtain $\bar{R} \in \SO(g_{x_0}, \e_d)$ such that
    \begin{equation} \label{eq:rigidity_applied_to_v_again}
        \int_Q |Dv - \bar{R}|_{g, \e_d}^p \, dx \le C \left(|Q| \left(\osc_Q g\right)^p + \int_Q \dist_{g, \e_d}^p(Dv, \SO(g, \e_d)) \, dx\right).
    \end{equation}
    Set $R := T \circ \bar{R}$. Then $R \in \SO((\mathbb{R}^d, g_{x_0}), (\Pi_0, \e_{d + 1}))$, and by \eqref{eq:error_in_projection} in Lemma \ref{lm:error_in_projection_and_distance} (see below), we have for a.e. $x \in F$,
    \begin{equation} \label{eq:change_from_v_to_u}
        \begin{aligned}
            |D\tilde{u}(x) - R|_{g, \e_{d + 1}} &\le |D\tilde{u}(x) - P D\tilde{u}(x)|_{g, \e_{d + 1}} + |P D\tilde{u}(x) - R|_{g, \e_{d + 1}} \\
            &\le |D\tilde{u}(x)|_{g, \e_{d + 1}} |n(x) - n(x_0)| + |Dv(x) - \bar{R}|_{g, \e_d} \\
            &\le \sqrt{d} |n(x) - n(x_0)| + 2 \dist_{g, \e_{d + 1}}(D\tilde{u}(x), \Ort(g_x, \e_{d + 1})) + |Dv(x) - \bar{R}|_{g, \e_d}
        \end{aligned}
    \end{equation}
    Furthermore, \eqref{eq:bound_for_distance_from_oriented_rotations} gives
    \begin{multline*}
        \dist_{g, \e_d}(Dv(x), \SO(g_x, \e_d)) = \dist_{g, \e_{d + 1}}(P D\tilde{u}(x), \SO((\mathbb{R}^d, g_x), (\Pi_0, \e_{d + 1})) \\
        \le \dist_{g, \e_{d + 1}}(D\tilde{u}(x), \Ort(g_x, \e_{d + 1})) + C|n(x) - n(x_0)|.
    \end{multline*}
    for a.e. $x \in F$. Hence, we arrive at
    \begin{equation} \label{eq:rigidity_on_good_set}
        \begin{aligned}
            \int_F |D\tilde{u} - R&|_{g, \e_{d + 1}}^p \, dx \\
            &\le C\left(|Q| \left(\osc_Q g\right)^p + \int_F \dist_{g, \e_{d + 1}}^p(D\tilde{u}, \Ort(g_, \e_{d + 1})) \, dx + \int_F |n(x) - n(x_0)|^p \, dx\right).
        \end{aligned}
    \end{equation}
    If we identify $T_x\mathbb{R}^d$ with $\mathbb{R}^d$, then $D\tilde{u}(x) = D\varphi(u(x)) \circ du_x$ for a.e. $x \in F$. Thus, part 1 of Lemma \ref{lm:non_euclidean_to_euclidean} implies
    \[
    \dist_{g, \e_{d + 1}}(D\tilde{u}(x), \Ort(g_x, \e_{d + 1})) \le \sqrt{1 + \varepsilon} \dist_{g, h}(du_x, \Ort((\mathbb{R}^d, g_x), (T_{u(x)}N, h_{u(x)}))) + C\varepsilon
    \]
    for a.e. $x \in F$. Consequently, we obtain from \eqref{eq:rigidity_on_good_set} that
    \begin{equation} \label{eq:rigidity_on_good_set_with_energy}
        \int_F |D\tilde{u} - R|_{g, \e_{d + 1}}^p \, dx \le C\left(|Q|\left(\varepsilon^p + \left(\osc_Q g\right)^p\right) + E_s(u) + \int_F |n(x) - n(x_0)|^p \, dx\right).
    \end{equation}
    To conclude the proof, we apply the triangle inequality in $Q \setminus F$:
    \begin{equation} \label{eq:rigidity_on_bad_set}
        \int_{Q \setminus F} |D\tilde{u} - R|_{g, \e_{d + 1}}^p \, dx \le C\left(\int_{Q \setminus F} (1 + \varepsilon)^\frac{p}{2}|du|_{g, h}^p + 1\, dx \right) \le C(E_s(u) + \delta^p|Q|).
    \end{equation}
    Here $C$ depends on $\theta$. Finally, combining \eqref{eq:euclidean_normal_variation}, \eqref{eq:rigidity_on_good_set_with_energy} and \eqref{eq:rigidity_on_bad_set} yields the desired inequality.    
\end{proof}

In the proof above, we used the following lemma to estimate the error due to projecting $\tilde{u}$ onto $\Pi_0$. A more general version of the lemma and its proof is given in \cite[Proposition 3.8]{B}.

\begin{lemma} \label{lm:error_in_projection_and_distance}
    Let $g_0$ be a constant metric on $\mathbb{R}^d$. Let $n_0, n \in \mathbb{R}^{d + 1}$ be unit vectors and set $\Pi := \lspan(n)^\perp$, $\Pi_0 := \lspan(n_0)^\perp$. Let $T \in \mathcal{L}(\mathbb{R}^d, \Pi)$ for some $x \in \mathbb{R}^d$ and $y \in \mathbb{R}^{d + 1}$. If $P$ is the orthogonal projection from $\mathbb{R}^{d + 1}$ onto $\Pi_0$, then
    \begin{equation} \label{eq:error_in_projection}
        |PT - T|_{g_0, \e_{d + 1}} \le |T|_{g_0, \e_{d + 1}} |n_0 - n|. 
    \end{equation}
    If $T$ is orientation preserving, then
    \begin{equation} \label{eq:bound_for_distance_from_oriented_rotations}
        \dist_{g_0, e_{d + 1}}(PT, \SO((\mathbb{R}^d, g_0), (\Pi_0, \e_{d + 1}))) \le \dist_{g_0, e_{d + 1}}(T, \Ort(g_0, \e_{d + 1})) + C|n_0 - n|.
    \end{equation}
\end{lemma}

\section{Asymptotic rigidity} \label{sec:asymptotic}

In the next theorem, we state our main result.

\begin{theorem} \label{thm:asymptotic_rigidity}
        Let $(M, g)$ and $(N, h)$ be oriented, connected, complete Riemannian manifolds of dimensions $d$ and $d + 1$, respectively. Let $M$ be compact. Assume $(u_k) \subset \Imm_p(M; N)$ satisfies
        \begin{equation} \label{eq:vanishing_and_bounded_energies}
            \lim_{k \to \infty} E_s(u_k) = 0, \quad \limsup_{k \to \infty} E_b(u_k) < \infty
        \end{equation}
        and
        \begin{equation} \label{eq:uniformly_bounded_in_lp}
            \limsup_{k \to \infty} \int_M d_h^p(u_k, q_0) \, d\vol_g < \infty
        \end{equation}
        for some $q_0 \in N$. Then there exists a subsequence $(u_{k_j})$ and $u \in \Imm_p(M; N)$ such that
        \begin{equation} \label{eq:convergence_to_local_isometry}
            u_{k_j} \to u \text{ in } \dot{W}^{1, p}(M; N), \quad du_x \in \Ort((T_xM, g_x), (T_{u(x)}N, h_{u(x)})) \text{ for a.e. } x \in M.
        \end{equation}
        Furthermore, if $S : TM \rightarrow TM$ is the reference shape operator with $|S|_g \in L^{p'}(M)$, where $p'$ is the Hölder conjugate of $p$, and
        \begin{equation}
            \lim_{k \to \infty} E_b^S(u_k) = 0,
        \end{equation}
        then $S = S_u$ a.e. in $M$.
\end{theorem}

We prove Theorem \ref{thm:asymptotic_rigidity} in Section \ref{sec:proof_of_asymptotic_rigidity}. In this section, we collect results that we will need in the proof. As before, we fix an open and bounded cube $Q \subset \mathbb{R}^d$ endowed with a metric $g$ such that
\[
\frac{1}{\lambda} \e_d \le g \le \lambda \e_d \quad \text{in } Q
\]
for some $\lambda > 0$. Since $g$ is comparable to the Euclidean metric, we will often bound quantities measured with respect to $g$ by those measured with respect to the Euclidean metric and vice versa. The result is summarized in the next lemma, which we will use implicitly most of the time.

\begin{lemma} \label{pr:change_in_volume}
    If $f : Q \rightarrow [0, \infty]$ is measurable, then
    \[
    \frac{1}{\lambda^{d/2}} \int_Q f \, dx \le \int_Q f \, d\vol_g \le \lambda^{d/2} \int_Q f \, dx.
    \]
    If $V$ is a vector space endowed with a constant metric $h_0$ and $L \in \mathcal{L}(\mathbb{R}^d, V)$, then
    \[
    \frac{1}{\sqrt{\lambda}} |L|_{g_x, h_0} \le |L|_{\e_d, h_0} \le \sqrt{\lambda} |L|_{g_x, h_0} \quad \text{for all } x \in Q.
    \]
\end{lemma}

\begin{proof}
    The proof of the first claim is elementary. The second claim follows from Lemma 3.4 in \cite{B}.
\end{proof}

Our main goal in this section is to prove that two Sobolev immersions are close in the $\dot{W}^{1, p}$-sense on small scales if their $L^p$-distance and stretching energies are small. We begin with a lemma that bounds the norm of a linear map by its ``$L^p$-variation".

\begin{lemma} \label{lm:norm_estimate}
    Let $R \in \mathcal{L}(\mathbb{R}^d; \mathbb{R}^{d + 1})$. Then there exists a constant $C$ depending only on $p$ and $d$ such that
    \[
    |R| \le C |Q|^{-\frac{1}{d} - \frac{2}{p}} \left(\int_Q \int_Q |R(x - y)|^p \, dx \, dy\right)^\frac{1}{p}.
    \]
\end{lemma}

\begin{proof}
    Without loss of generality, we assume that $Q$ is the unit cube centered at $0$. For more general cubes, the inequality follows by a scaling argument. Define $\Psi : \mathcal{L}(\mathbb{R}^d; \mathbb{R}^{d + 1}) \rightarrow \mathbb{R}$ by
    \[
    \Psi(R) = \left(\int_Q \int_Q |R(x - y)|^p \, dx \, dy\right)^\frac{1}{p}.
    \]
    Clearly, $\Psi$ is a continuous function on $\mathcal{L}(\mathbb{R}^d; \mathbb{R}^{d + 1})$. Set $m := \min \{\Psi(R) : |R| = 1\}$. Since $\Psi(R) = 0$ if and only if $R = 0$, we have $m > 0$. Noting that both $|R|$ and $\Psi(R)$ are positively homogeneous of degree 1, we have
    \[
    |R| \le \frac{1}{m} \left(\int_Q \int_Q |R(x - y)|^p \, dx \, dy\right)^\frac{1}{p} \text{for all } R \in \mathcal{L}(\mathbb{R}^d; \mathbb{R}^{d + 1}).
    \]
\end{proof}

Let $(N, h)$ be a connected, oriented Riemannian manifold. We denote the Sasaki metric on $(T^*Q \otimes TN, \pi, Q \times N)$ by $\sigma$. In the next lemma, we prove an estimate for the Sasaki distance of linear maps in terms of their distance in coordinates.

\begin{lemma} \label{lm:sasaki_distance_estimate}
    Let $(U, \varphi)$ be an $\varepsilon$-isometric chart on $N$. Assume that $B(0, r) \subset \varphi(U)$. Let $x \in Q$ and let $q_i \in \varphi^{-1}(B(0, r))$, $L_i \in \mathcal{L}(T_x\mathbb{R}^d, T_{q_i}N)$ for $i = 1, 2$. Then
    \[
    d_\sigma(L_1, L_2) \le C(|D\varphi(q_1) \circ L_1 - d\varphi(q_2) \circ L_2|_{g, \e_{d + 1}} + d_h(q_1, q_2)(1 + |L_2|_{g, h})),
    \]
    where $C$ depends only on $d$.
\end{lemma}

\begin{proof}
    Define $y(t) := \varphi(q_1) + t(\varphi(q_2) - \varphi(q_1))$ for $t \in [0, 1]$ and set $\gamma(t) := \varphi^{-1}(y(t))$. By Proposition \ref{pr:sasaki_distance} we know that
    \begin{equation} \label{eq:sasaki_distance_estimate}
        d_\sigma(L_1, L_2)^2 \le |L_1 - L_\gamma|_{g, h}^2 + \left(\int_0^1 |\gamma'|_h \, dt\right)^2,
    \end{equation}
    where $L_\gamma \in \mathcal{L}(T_x\mathbb{R}^d, T_{q_1}N)$ is defined as in \eqref{eq:linear_map_transport}. We can bound the second term on the right-hand side simply by
    \begin{equation} \label{eq:arclength_estimate}
        \int_0^1|\gamma'(t)|_h \, dt = \int_0^1 |d(\varphi^{-1})_{y(t)}(\varphi(q_2) - \varphi(q_1))|_h \, dt \le 2 |\varphi(q_2) - \varphi(q_1)|.
    \end{equation}
    Let $v \in T_x\mathbb{R}^d$. We define $\alpha(t) \in T_{\gamma(t)}N$ to be the vector parallel to $L_2v$ along $\gamma$. We shall estimate $|(L_1 - L_\gamma)v|_h = |L_1v- \alpha(0)|_h$. If we represent $\alpha$ in coordinates by
    \[
    \alpha(t) = \sum_{k = 1}^d a_k(t) \frac{\partial}{\partial y_k}\bigg|_{\gamma(t)},
    \]
    then $a_1, \dots, a_d$ satisfy the following system of ordinary differential equations:
    \begin{equation} \label{eq:parallel_transport_ode}
        a_k'(t) = - \sum_{i, j = 1}^d a_i(t) y_j'(t) \Gamma_{i, j}^k(\gamma(t)), \quad k = 1, \dots, d,
    \end{equation}
    where $\Gamma_{i,j }^k : U \rightarrow \mathbb{R}$ are the Christoffel symbols. The triangle inequality gives
    \begin{equation} \label{eq:comparison_with_parallel_transport}
        |L_1v - \alpha(0)|_h \le 2|D\varphi(q_1)(L_1v - \alpha(0))| \le 2|D\varphi(q_1)(L_1v) - D\varphi(q_2)(L_2v)| + 2|a(1) - a(0)|.
    \end{equation}
    By \eqref{eq:epsilon_isometric_chart} and \eqref{eq:parallel_transport_ode}, we have
    \begin{multline*}
        |a(1) - a(0)| \le \int_0^1 |a'(t)| \, dt \le C \int_0^1 \max_{i, j, k} \|\Gamma_{i, j}^k\|_{L^\infty(U)} |a(t)| |y'(t)| \, dt \\
        \le C \int_0^1 |a(t)| |y'(t)| \, dt \le C |\varphi(q_2) - \varphi(q_1)| \int_0^1 |\alpha(t)|_h \, dt.
    \end{multline*}
    Since parallel transport preserves length, we know that $|\alpha(t)|_h = |L_2v|_h$. Hence,
    \[
    |a(1) - a(0)| \le C|\varphi(q_2) - \varphi(q_1)||L_2v|_h.
    \]
    Thus, by \eqref{eq:comparison_with_parallel_transport}, we conclude that
    \[
    |(L_1 - L_\gamma)v|_h \le C(|D\varphi(q_1)(L_1v) - D\varphi(q_2)(L_2v)| + |\varphi(q_1) - \varphi(q_2)||L_2v|_h).
    \]
    Consequently, we obtain
    \[
    |L_1 - L_\gamma|_{g, h} \le C(|D\varphi(q_1) \circ L_1 - D\varphi(q_2) \circ L_2|_{g, \e_{d + 1}} + |\varphi(q_1) - \varphi(q_2)||L_2|_{g, h})
    \]
    Finally, \eqref{eq:sasaki_distance_estimate}, \eqref{eq:arclength_estimate} and the inequality $|\varphi(q_1) - \varphi(q_2)| \le 2d_h(q_1, q_2)$ yield the claim.
\end{proof}

We come to the main result of this section.

\begin{proposition} \label{pr:sasaki_distance_in_lp}
    Let $(U, \varphi)$ be an $\varepsilon$-isometric chart on $N$ and let $\varphi^{(r)}$ be its extension by $\theta$. Let $u_1, u_2 \in \Imm_p(Q; N)$. Let $\delta \in (0, 1)$ and assume that $F_k \subset Q$ is a Lebesgue measurable set such that
    \[
    u_k(F_k) \subset \varphi^{-1}(B(0, r)), \quad \frac{|Q \setminus F_k|}{|Q|} \le \delta^p \quad \text{for } k = 1, 2.
    \]
    Then
    \begin{multline} \label{eq:reverse_poincaré}
        \int_Q d_\sigma^p(u_1, u_2) \, dx \le C|Q|\left(\delta^p + \varepsilon^p + \left(\osc_Q g\right)^p + \frac{\diam^p(Q)}{1 - \delta^p} \left(1 + \frac{\delta^p}{r^p}\right)\right)\\
        + C\left(\left(1 + \frac{1}{r^p}\right) \max_{k =1, 2} E_s(u_k) + \frac{\diam^p(Q)}{1 - \delta^p} \max_{k =1, 2} E_b(u_k) + \left(1 + |Q|^{-\frac{p}{d}}\right)\int_Q d_h^p(u_1, u_2) \, dx\right).
    \end{multline}
    The constant $C$ depends only on $d$, $p$, $\lambda$ and $\theta$.
\end{proposition}

\begin{proof}
    We briefly summarize the content and the proof method of the lemma. The reader should consider $E_s(u_1)$, $E_s(u_2)$ and $|Q|$ as sufficiently small quantities, so that \eqref{eq:reverse_poincaré} implies that the $\dot{W}^{1,p}$ distance of $u_1$ and $u_2$ is controlled by their $L^p$ distance. We shall prove this by using Theorem \ref{thm:local_quantitative_rigidity} to show that $u_1$ and $u_2$ are close to some rotations $R_1$ and $R_2$ in coordinates. Then knowing the $L^p$ distance of $u_1$ and $u_2$ gives us control on $|R_1 - R_2|$, which controls the distance of $du_1$ and $du_2$ thanks to the rigidity theorem.
    
    Set $\tilde{u}_k := \varphi^{(r)} \circ u_k$. By the quantitative rigidity estimate in Theorem \ref{thm:local_quantitative_rigidity}, there exist $x_k \in Q$ and $R_k \in \Ort(g_x, \e_{d + 1})$ for $k = 1, 2$ such that
    \begin{equation} \label{eq:rigidity_applied_to_sequence}
        \int_Q |D\tilde{u}_k - R_k|_{g, \e_{d + 1}}^p \, dx \le C\left(|Q|\left(\delta^p + \varepsilon^p + \left(\osc_Q g\right)^p\right) + E_s(u_k) + \frac{\diam^p(Q)}{1 - \delta^p} \mathcal{E}(u_k)\right).
    \end{equation}
    For $x \in F_k$, we define $L_k(x) \in  \mathcal{L}(T_x\mathbb{R}^d, T_{u_k(x)}N)$ by $L_k(x) := d(\varphi^{-1})_{\tilde{u}_k(x)} \circ R_k$, where we identify $T_{\tilde{u}_k(x)}\mathbb{R}^{d + 1}$ with $\mathbb{R}^{d + 1}$. Since $D\varphi \circ du_k = D\tilde{u}_k$ a.e. in $F_k$, we have
    \begin{equation} \label{eq:true_distance_estimate}
        |d(u_k)_x - L_k(x)|_{g, h} \le 2 |D\varphi(u_k(x)) (d(u_k)_x - L_k(x))|_{g, \e_{d + 1}} = 2|D\tilde{u}_k(x) - R_k|_{g, \e_{d + 1}}.
    \end{equation}
    for a.e. $x \in F_k$. Furthermore, it follows from Lemma \ref{lm:sasaki_distance_estimate} that
    \begin{multline*}
        d_\sigma(L_1(x), L_2(x)) \le C(|D\varphi(u_1(x)) \circ L_1(x) - D\varphi(u_2(x)) \circ L_2(x)|_{g, \e_{d + 1}} + d_h(u_1(x), u_2(x))(1 + |L_2(x)|_{g, h})) \\
        \le C(|R_1 - R_2|_{g, \e_{d + 1}} + d_h(u_1(x), u_2(x))(1 + 2|R_2|_{g, h})) \le C(|R_1 - R_2|_{g, \e_{d + 1}} + d_h(u_1(x), u_2(x))),
    \end{multline*}
    where the last constant depends on $d$ and $\lambda$. Using the triangle inequality and Proposition \ref{pr:triangle_inequality}, we obtain for a.e. $x \in F_1 \cap F_2$,
    \begin{multline} \label{eq:pointwise_sasaki_distance}
        d_\sigma(d(u_1)_x, d(u_2)_x) \le d_\sigma(d(u_1)_x, L_1(x)) + d_\sigma(L_1(x), L_2(x)) + d_\sigma(L_2(x), d(u_2)_x) \\
        \le |d(u_1)_x - L_1(x)|_{g, h} + |d(u_2)_x - L_2(x)|_{g, h} + C(|R_1 - R_2|_{g, \e_{d + 1}} + d_h(u_1(x), u_2(x))).
    \end{multline}
    Integrating \eqref{eq:pointwise_sasaki_distance} over $F_1 \cap F_2$ and using \eqref{eq:true_distance_estimate} gives
    \begin{equation} \label{eq:sasaki_distance_differentials}
        \int_{F_1 \cap F_2} d_\sigma^p(du_1, du_2) \, dx \le C\int_Q |D\tilde{u}_1 - R_1|_{g, \e_{d + 1}}^p + |D\tilde{u}_2 - R_2|_{g, \e_{d + 1}}^p + |R_1 - R_2|_{g, \e_{d + 1}}^p + d_h^p(u_1, u_2) \, dx.
    \end{equation}
    Next, we estimate $|R_1 - R_2|$. By Poincaré's inequality applied to $\tilde{u}_k - R_k$, we have
    \begin{equation} \label{eq:poincaré_with_rotation}
        \int_Q \int_Q |\tilde{u}_k(x) - \tilde{u}_k(y) - R_k(x - y)|^p \, dx \, dy \le C \diam^p(Q) |Q| \int_Q |D\tilde{u}_k(x) - R_k|^p \, dx \quad \text{for } k = 1, 2.
    \end{equation}
    Since
    \[
    |(R_1 - R_2)(x - y)| \le |\tilde{u}_1(x) - \tilde{u}_1(y) - R_1(x - y)| + |\tilde{u}_2(x) - \tilde{u}_2(y) - R_2(x - y)| + |\tilde{u}_1(x) - \tilde{u}_2(x)| + |\tilde{u}_1(y) - \tilde{u}_2(y)|,
    \]
    the inequality \eqref{eq:poincaré_with_rotation} implies
    \begin{multline*}
        \int_Q \int_Q |(R_1 - R_2)(x - y)|^p \, dx \, dy \\
        \le C |Q|\left(\diam^p(Q) \int_Q |D\tilde{u}_1(x) - R_1|^p + |D\tilde{u}_2(x) - R_2|^p \, dx + \int_Q |\tilde{u}_1(x) - \tilde{u}_2(x)|^p \, dx \right).
    \end{multline*}
    Finally, Lemma \ref{lm:norm_estimate} yields
    \begin{equation} \label{eq:difference_of_rotations}
        \begin{aligned}
            &|R_1 - R_2|^p \le |Q|^{-\frac{p}{d} - 2} \int_Q \int_Q |(R_1 - R_2)(x - y)|^p \, dx \, dy \\
            &\le C |Q|^{-\frac{p}{d} - 1} \left(\diam^p(Q) \int_Q |D\tilde{u}_1(x) - R_1|^p + |D\tilde{u}_2(x) - R_2|^p \, dx + \int_Q |\tilde{u}_1(x) - \tilde{u}_2(x)|^p \, dx \right) \\
            &\le C|Q|^{-1} \left(\int_Q |D\tilde{u}_1(x) - R_1|^p + |D\tilde{u}_2(x) - R_2|^p \, dx + |Q|^{-\frac{p}{d}}\int_Q |\tilde{u}_1(x) - \tilde{u}_2(x)|^p \, dx \right).
        \end{aligned}
    \end{equation}
    We observe that there exists a constant $C$ depending only on $\theta$ such that 
    \[
    \|\tilde{u}_1 - \tilde{u}_2\|_{L^\infty(Q)} \le C\|d_h(u_1, u_2)\|_{L^\infty(Q)} \quad \text{ for all } k, l \in \mathbb{N}.
    \]
    Hence, we obtain
    \begin{multline*}
        \int_Q |R_1 - R_2|_{g, \e_{d + 1}}^p \, dx \le C |Q| |R_1 - R_2|^p \\
        \le C \left(\int_Q |D\tilde{u}_1 - R_1|_{g, \e_{d + 1}}^p + |D\tilde{u}_2 - R_2|_{g, \e_{d + 1}}^p \, dx + |Q|^{-\frac{p}{d}}\int_Q d_h^p(u_1, u_2) \, dx \right)
    \end{multline*}
    The previous estimate, in combination with \eqref{eq:sasaki_distance_differentials}, implies
    \begin{equation}
        \int_{F_1 \cap F_2} d_\sigma^p(du_1, du_2) \, dx \le C\int_Q |D\tilde{u}_1 - R_1|_{g, \e_{d + 1}}^p + |D\tilde{u}_2 - R_2|_{g, \e_{d + 1}}^p + (1 + |Q|^{-\frac{p}{d}})d_h^p(u_1, u_2) \, dx.
    \end{equation}
    Therefore, \eqref{eq:rigidity_applied_to_sequence} gives
    \begin{multline*}
        \int_{F_1 \cap F_2} d_\sigma^p(du_1, du_2) \, dx \le C|Q|\left(\delta^p + \varepsilon^p + \left(\osc_Q g\right)^p + \right) + C \max_{k =1, 2} E_s(u_k)\\
        + C\left(\frac{\diam^p(Q)}{1 - \delta^p} \max_{k =1, 2} \mathcal{E}(u_k) + \left(1 + |Q|^{-\frac{p}{d}}\right)\int_Q d_h^p(u_1, u_2) \, dx\right).
    \end{multline*}
    Since,
    \begin{align*}
        \mathcal{E}(u_k) &\le C\left(\frac{1}{r^p} |Q \setminus F_k| + |Q| + \left(1 + \frac{1}{r^p}\right) E_s(u_k) + E_b(u_k) \right) \\
        &\le C\left(\left(1 + \frac{\delta^p}{r^p}\right)|Q| + \left(1 + \frac{1}{r^p}\right) E_s(u_k) + E_b(u_k) \right) \quad \text{for } k = 1, 2,
    \end{align*}
    we arrive at
    \begin{multline} \label{eq:estimate_on_good_set}
        \int_{F_1 \cap F_2} d_\sigma^p(du_1, du_2) \, dx \le C|Q|\left(\delta^p + \varepsilon^p + \left(\osc_Q g\right)^p + \frac{\diam^p(Q)}{1 - \delta^p} \left(1 + \frac{\delta^p}{r^p}\right)\right)\\
        + C\left(\left(1 + \frac{1}{r^p}\right) \max_{k =1, 2} E_s(u_k) + \frac{\diam^p(Q)}{1 - \delta^p} \max_{k =1, 2} E_b(u_k) + \left(1 + |Q|^{-\frac{p}{d}}\right)\int_Q d_h^p(u_1, u_2) \, dx\right).
    \end{multline}
    On the other hand, we can apply the last point in Proposition \ref{pr:triangle_inequality} outside of $F_1 \cap F_2$ to get
    \begin{multline} \label{eq:crude_estimate_on_bad_set}
        \int_{Q \setminus (F_1 \cap F_2)} d_\sigma^p(du_1, du_2) \, dx \le C \int_{Q \setminus (F_1 \cap F_2)} |du_1|_{g, h}^p + d_h^p(u_1, u_2) + |du_2|_{g, h}^p \, dx \\
        \le C\left(\delta^p|Q| + E_s(u_1) + E_s(u_2) + \int_Q d_h^p(u_1, u_2) \, dx \right).
    \end{multline}
    To conclude the proof, we simply add \eqref{eq:estimate_on_good_set} and \eqref{eq:crude_estimate_on_bad_set}.
\end{proof}

\section{Proof of Theorem \ref{thm:asymptotic_rigidity}} \label{sec:proof_of_asymptotic_rigidity}

Since $M$ is compact, there is no loss of generality in proving the statement in a single chart. Let $(W, \psi)$ be a chart on $M$ such that $Q := \psi(W)$ is an open and bounded cube. We denote the pushforward of the metric $g$ by $\psi$ by the same letter and assume without loss of generality that
\[
\frac{1}{\lambda} \e_d \le g \le \lambda \e_d \quad \text{in } Q
\]
for some $\lambda > 0$ and that $g$ is uniformly continuous in $Q$. We denote the coordinate version of $u_k$ and $\nu_{u_k}$ by the same letter. We briefly outline the steps of the proof. In step 1, we show that the sequences $(u_k)$ and $(\nu_{u_k})$ have converging subsequences in $L^p$. In steps 2 and 3, we partition $Q$ into small cubes and use the convergence of $(u_k)$ and Proposition \ref{pr:sasaki_distance_in_lp} to show that $(du_k)$ has a Cauchy subsequence in $\dot{W}^{1, p}$. In step 4, we prove that the limit function we identified is actually a Sobolev immersion, and finally, in step 5, we prove that if $E_b^S(u_k) \to 0$, then the induced shape operator of the limit agrees with $S$, the reference shape operator.

\underline{Step 1:} To begin with, we shall apply Theorem \ref{thm:rellich_kondrachov} to the sequences $(u_k)$ and $(\nu_{u_k})$ in order to pass to converging subsequences. For simplicity, we write $\nu_k$ for $\nu_{u_k}$. Using the triangle inequality, we get
\begin{equation} \label{eq:uniformly_bounded_derivatives}
    \int_Q |du_k|_{g, h}^p \, dx \le C(|Q| + E_s(u_k)).
\end{equation}
Hence, it is clear from \eqref{eq:vanishing_and_bounded_energies} and \eqref{eq:uniformly_bounded_in_lp} that $(u_k)$ satisfies the hypotheses of Theorem \ref{thm:rellich_kondrachov}. If we denote the zero vector of $T_qN$ by $0_q$, we have
\[
d_\sigma(\nu_k(x), 0_{q_0}) \le d_\sigma(\nu_k(x), 0_{u_k(x)}) + d_\sigma(0_{u_k(x)}, 0_{q_0}) = |\nu_k(x)|_h + d_h(u_k(x), q_0) \quad \text{for all } x \in Q,
\]
so that
\begin{equation} \label{eq:normals_uniformly_bounded_in_lp}
    \int_Q d_\sigma^p(\nu_k, 0_{q_0}) \, dx \le C\left(|Q| + \int_Q d_h^p(u_k, q_0) \, dx\right).
\end{equation}
On the other hand, by the definition of the Sasaki metric, we know that
\[
|d(\nu_k)_x|_{g, \sigma}^2 = |K_{TN} \circ d(\nu_k)_x|_{g, h}^2 + |d(\pi_N)_{u(x)} \circ d(\nu_k)_x|_{g, h}^2. 
\]
Since $d\pi_N \circ d\nu_k = du_k$ a.e. in $Q$, we obtain
\[
\int_Q |d\nu_k|_{g, \sigma}^2 \, dx \le C\left(E_b(u_k) + \int_Q |du_k|_{g, h}^p \, dx\right).
\]
Thus, by \eqref{eq:vanishing_and_bounded_energies}, \eqref{eq:uniformly_bounded_derivatives} and \eqref{eq:normals_uniformly_bounded_in_lp}, we can apply Theorem \ref{thm:rellich_kondrachov} to $(\nu_k)$ as well. As a result, there exist subsequences of $(u_k)$ and $(\nu_k)$, not relabeled, and there exist functions $u \in \dot{W}^{1, p}(Q; N)$, $\nu \in \dot{W}^{1, p}(Q; TN)$ such that 
\begin{equation} \label{eq:lp_and_pointwise_convergence}
    \begin{aligned}
        d_h(u_k, u) \to 0 \text{ in } L^p(Q), &\quad u_k \to u \text{ pointwise a.e. in } Q, \\
        d_\sigma(\nu_k, \nu) \to 0 \text{ in } L^p(Q), &\quad \nu_k \to \nu \text{ pointwise a.e. in } Q.
    \end{aligned}
\end{equation}
Later, we will show that $u \in \Imm_p(Q; N)$ and $\nu(x) = \nu_u(x)$ a.e. in $Q$. We also note a consequence of Lemma \ref{lm:closure_under_pointwise_convergence} that we shall need later: for all open sets $D \subset Q$, we have
\begin{multline} \label{eq:bounded_average_integral_norm}
    \int_D |du|_{g, h}^p \, dx \le C \int_{\psi^{-1}(D)} |du|_{g, h}^p \, d\vol_g\\
    \le C \liminf_{k \to \infty} \int_{\psi^{-1}(D)} |du_k|_{g, h}^p \, d\vol_g \le \liminf_{k \to \infty} C (\vol_g(\psi^{-1}(D)) + E_s(u_k)) \le C|D|.
\end{multline}
\underline{Step 2:} We fix $\varepsilon \in (0, 1)$ and choose $\tau > 0$ such that $|Q \cap u^{-1}(B(q_0, \tau))| \le \varepsilon^p |Q|$. For every $q \in B(q_0, \tau)$, there exists an $\varepsilon$-isometric chart $(U, \varphi)$ such that $\varphi(q) = 0$. For every such chart, we pick $r, \rho \in (0, 1)$ such that $B(q, 3\rho) \subset \varphi^{-1}(B(0, r))$ and $\varphi^{-1}(B(0, 2r)) \subset U$. Since $N$ is complete, the Hopf-Rinow theorem implies that $B(q_0, \tau)$ is relatively compact. Hence, we can find finitely many points $q_1, \dots, q_K \in N$ and associated radii $r_i, \rho_i$ for $i = 1, \dots, K$ such that
\[
B(q_0, \tau) \subset \bigcup_{i = 1}^K B(q_i, \rho_i).
\]
Set $r_0 := \min \{r_1, \dots, r_K\}$ and $\rho_0 := \min \{\rho_1, \dots, \rho_K\}$. Denote the side length of $Q$ by $s$. We partition $Q$ into identical cubes of side length $s/m$. We define $\mathcal{G}_m$ to be the set of cubes $Q'$ in the partition which satisfy
\begin{equation} \label{eq:density_bound}
    \frac{|Q' \cap u^{-1}(B(q_0, \tau))|}{|Q'|} > \frac{1}{2}.
\end{equation}
Then $\mathcal{B}_m$ is defined to be the set of remaining cubes in the partition. Since $|Q \cap u^{-1}(B(q_0, \tau))| \le \varepsilon^p |Q|$, it is easy to see that
\begin{equation} \label{eq:small_measure}
    \left|\bigcup_{Q' \in \mathcal{B}_m} Q'\right| \le 2\varepsilon^p |Q|.
\end{equation}
We claim that for all sufficiently large $m$, we can find $k(m) \in \mathbb{N}$ such that if $Q' \in \mathcal{G}_m$, then there exists $j \in \{1, \dots, K\}$ for which
\begin{equation} \label{eq:claim_about_density_bound}
    \frac{|Q' \setminus u_k^{-1}(B(q_j, 3\rho_j))|}{|Q'|} \le \varepsilon^p \quad \text{for all } k > k(m).
\end{equation}
Fix $Q' \in \mathcal{G}_m$. By Poincaré's inequality for manifold valued maps given in Lemma \ref{lm:poincare_manifold}, we have
\[
\int_{Q'} \int_{Q'} d_h^p(u(x), u(z)) \, dx \, dz \le C \diam^p(Q') |Q'| \int_{Q'} |du|_{\e_d, h}^p \, dx.
\]
By Chebyshev's inequality and \eqref{eq:density_bound}, there exists a point $x_0 \in Q' \cap u^{-1}(B(q_0, \tau))$ such that
\[
\int_{Q'} d_h^p(u(x), u(x_0)) \, dx \le 2C \diam^p(Q') \int_{Q'} |du|_{\e_d, h}^p \, dx.
\]
Applying Chebyshev's inequality again, we find that
\begin{equation} \label{eq:density_estimate}
    \frac{|Q' \setminus u^{-1}(B(u(x_0), \rho_0))|}{|Q'|} \le C \left(\frac{\diam(Q')}{\rho_0}\right)^p \frac{1}{|Q'|} \int_{Q'} |du|_{\e_d, h}^p \, dx
\end{equation}
It follows from \eqref{eq:bounded_average_integral_norm} that
\[
\int_{Q'} |du|_{\e_d, h}^p \, dx \le C|Q'|,
\]
where $C$ depends only on $d$ and $\lambda$. Thus, \eqref{eq:density_estimate} becomes
\begin{equation} \label{eq:density_estimate_2}
    \frac{|Q' \setminus u^{-1}(B(u(x_0), \rho_0))|}{|Q'|} \le C \left(\frac{\diam(Q')}{\rho_0}\right)^p \le \frac{C}{(m \rho_0)^p}.
\end{equation}
Since $u(x_0) \in B(q_0, \tau)$, there exists $j \in \{1, \dots, K\}$ such that $u(x_0) \in B(q_j, \rho_j)$. For $k \in \mathbb{N}$, it is easy to see that
\begin{align*}
|Q' \setminus u_k^{-1}(B(q_j, 3\rho_j))| &\le |Q' \setminus u_k^{-1}(B(u(x_0), 2q_0))| \\
&\le |Q' \setminus u^{-1}(B(u(x_0), \rho_0))| + |\{x \in Q': d_h(u_k(x), u(x)) \ge \rho_0\}|.
\end{align*}
Therefore, with the help of \eqref{eq:density_estimate_2} and Chebyshev's inequality, we reach
\[
\frac{|Q' \setminus u_k^{-1}(B(q_j, 3\rho_j))|}{|Q'|} \le \frac{C}{\rho_0^p}\left(\frac{1}{m^p} + m^d \int_{Q'} d_h^p(u_k, u) \, dx\right).
\]
Our claim now follows easily from \eqref{eq:lp_and_pointwise_convergence}. 

\underline{Step 3:} We fix a large $m$ and pick $k(m)$ such that for all $Q' \in \mathcal{G}_m$ there exists $j \in \{1, \dots, K\}$ satisfying \eqref{eq:claim_about_density_bound}. Then for $k_1, k_2 > k(m)$, we can apply Proposition \ref{pr:sasaki_distance_in_lp} to $u_{k_1}$ and $u_{k_2}$ in each $Q' \in \mathcal{G}_m$ and sum over all $Q'$ to obtain
\begin{multline*}
    \sum_{Q' \in \mathcal{G}_m} \int_{Q'} d_\sigma^p(du_{k_1}, du_{k_2}) \, dx \le C|Q|\left(\varepsilon^p + \max_{Q' \in \mathcal{G}_m}\left(\osc_{Q'} g\right)^p + \frac{1}{m^p(1 - \varepsilon^p)} \left(1 + \frac{\varepsilon^p}{r_0^p}\right)\right) \\
    + C\left(\left(1 + \frac{1}{r_0^p}\right) \sup_{k > k(m)} E_s(u_k) + \frac{1}{m^p(1 - \varepsilon^p)} \sup_{k > k(m)} E_b(u_k) + \left(1 + m^p\right)\int_Q d_h^p(u_{k_1}, u_{k_2}) \, dx\right)
\end{multline*}
for $k_1, k_2 > k(m)$. If we let $k_1, k_2 \to \infty$ and use \eqref{eq:vanishing_and_bounded_energies} and \eqref{eq:lp_and_pointwise_convergence}, then we see that
\begin{multline*}
    \limsup_{k_1, k_2 \to \infty} \sum_{Q' \in \mathcal{G}_m} \int_{Q'} d_\sigma^p(du_{k_1}, du_{k_2}) \, dx \\
    \le C\left(|Q|\left(\varepsilon^p + \max_{Q' \in \mathcal{G}_m} \left(\osc_{Q'} g\right)^p + \frac{1}{m^p(1 - \varepsilon^p)} \left(1 + \frac{\varepsilon^p}{r_0^p}\right)\right) + \frac{1}{m^p(1 - \varepsilon^p)} \limsup_{k \to \infty} E_b(u_k)\right).
\end{multline*}
On the other hand, if $Q' \in \mathcal{B}_m$, then we can apply the last point in Proposition \ref{pr:triangle_inequality} to get
\begin{multline*}
    \sum_{Q' \in \mathcal{B}_m} \int_{Q'} d_\sigma^p(du_{k_1}, du_{k_2}) \, dx \le \sum_{Q' \in \mathcal{B}_m}\int_{Q'} |du_{k_1}|_{g, h}^p + d_h^p(u_{k_1}, u_{k_2}) + |du_{k_2}|_{g, h}^p \, dx \\
    \le C\left|\bigcup_{Q' \in \mathcal{B}_m}Q'\right| +C \left(E_s(u_{k_1}) + E_s(u_{k_2}) + \int_Q d_h^p(u_{k_1}, u_{k_2}) \, dx\right).
\end{multline*}
Hence, using \eqref{eq:vanishing_and_bounded_energies}, \eqref{eq:lp_and_pointwise_convergence} and \eqref{eq:small_measure}, we arrive at
\[
\limsup_{k_1, k_2 \to \infty} \sum_{Q' \in \mathcal{B}_m} \int_{Q'} d_\sigma^p(du_{k_1}, du_{k_2}) \, dx \le C\varepsilon^p |Q|.
\]
Thus, we conclude that
\begin{multline*}
    \limsup_{k_1, k_2 \to \infty} \int_Q d_\sigma^p(du_{k_1}, du_{k_2}) \, dx \\
    \le C\left(|Q|\left(\varepsilon^p + \max_{Q' \in \mathcal{G}_m}\left(\osc_{Q'} g\right)^p + \frac{1}{m^p(1 - \varepsilon^p)} \left(1 + \frac{\varepsilon^p}{r_0^p}\right)\right) + \frac{1}{m^p(1 - \varepsilon^p)} \limsup_{k \to \infty} E_b(u_k)\right).
\end{multline*}
Finally, since $m$ and $\varepsilon$ were arbitrary, we first let $m \to \infty$ and then $\varepsilon \to 0$ to get
\[
\limsup_{k_1, k_2 \to \infty} \int_Q d_\sigma^p(u_{k_1}, u_{k_2}) \, dx = 0.
\]
Hence, by \eqref{eq:lp_and_pointwise_convergence} and Proposition \ref{pr:completeness_of_sobolev_space}, $u_k \to u$ in $\dot{W}^{1, p}(Q; N)$. 

\underline{Step 4:} We assume without loss of generality that $(du_k)$ converges in $T^*Q \otimes TN$ pointwise a.e. in $Q$. Since $E_s(u_k) \to 0$, it follows that $du_x \in \Ort((T_x\mathbb{R}^d, g_x), (T_{u(x)}N, h_{u(x)}))$ for a.e. $x \in Q$. To prove that $\nu = \nu_u$ a.e. in $Q$ suppose that $(d(u_k)_x)$ and $(\nu_k(x))$ both converge at $x \in Q$ and let $(v_1, \dots, v_d) \in T_x\mathbb{R}^d$ be a positively oriented orthonormal basis with respect to $g_x$. Then
\[
0 = \lim_{k \to \infty} (d(u_k)_x(v_i), \nu_k(x))_h = (du_x(v_i), \nu(x))_h \quad \text{for } i = 1, \dots, d.
\]
Thus, $(\nu(x), du_x(v))_h = 0$ for all $v \in T_x\mathbb{R}^d$. Furthermore, since $(d(u_k)_x(v_1), \dots, d(u_k)_x(v_d), \nu_k(x))$ is positively oriented for all $k$, so is $(du_x(v_1), \dots, du_x(v_d), \nu(x))$. Therefore, by \eqref{eq:lp_and_pointwise_convergence}, we conclude that $\nu = \nu_u$ a.e. in $Q$. Consequently, $u \in \Imm_p(Q; N)$.

\underline{Step 5:} For the final step of the proof, we assume $E_b^S(u_k) \to 0$. We first show that
\begin{equation} \label{eq:annoying_convergence}
    \lim_{k \to \infty} \int_Q d_{\sigma}(du_k \circ S_{u_k}, du \circ S) \, dx = 0.
\end{equation}
By triangle inequality, we observe that
\begin{align*}
    d_{\sigma}(d(u_k)_x \circ S_{u_k}(x), du_x \circ S(x)) &\le d_\sigma(d(u_k)_x \circ S_{u_k}(x), d(u_k)_x \circ S(x)) + d_\sigma(d(u_k)_x \circ S, du_x \circ S(x)) \\
    &= |d(u_k)_x \circ (S_{u_k}(x) - S(x))|_{g, h} + d_\sigma(d(u_k)_x \circ S(x), du_x \circ S(x))
\end{align*}
for all $x \in Q$. We claim that
\begin{equation} \label{eq:estimate_that_should_have_been_easy}
    d_\sigma(d(u_k)_x \circ S(x), du_x \circ S(x)) \le d_\sigma(d(u_k)_x, du_x)(1 +  |S(x)|_g).
\end{equation}
Assuming this for the moment, Hölder's inequality yields
\[
\int_Q d_{\sigma}(du_k \circ S_{u_k}, du \circ S) \, dx \le C\left(|Q|^\frac{1}{p'}E_b^S(u_k)^\frac{1}{p} + \left(\int_Q d_\sigma^p(du_k, du) \, dx\right)^\frac{1}{p}\left(\int_Q (1 + |S|_g)^{p'} \, dx\right)^\frac{1}{p'}\right).
\]
Hence, \eqref{eq:annoying_convergence} follows. To prove \eqref{eq:estimate_that_should_have_been_easy}, we use Corollary \ref{co:maps_sasaki_distance} and Lemma \ref{lm:map_parallel_transport}. Fix $x \in Q$. Let $\gamma : [0, 1] \rightarrow N$ be a piecewise regular curve with $\gamma(0) = u_k(x)$ and $\gamma(1) = u(x)$. Set $E := T^*Q \otimes TN$. Referring to the notation of Lemma \ref{lm:map_parallel_transport}, we have $P_E^\gamma(du_x \circ S(x)) = P_E^\gamma(du_x) \circ S(x)$. Therefore,
\[
|d(u_k)_x \circ S(x) - P_E^\gamma(du_x \circ S(x))|_{g, h}^2 = |(d(u_k)_x - P_E^\gamma(du_x)) \circ S(x)|_{g, h}^2 \le |d(u_k)_x - P_E^\gamma(du_x)|_{g, h}^2 |S(x)|_g^2.
\]
Since $\gamma$ was arbitrary, Corollary \ref{co:maps_sasaki_distance} proves the claim. Next, we will show that $du \circ S_u = du \circ S$ a.e. in $Q$. Given this, it follows from the almost everywhere injectivity of $du$ that $S = S_u$ a.e. in $Q$. Let $\eta \in C_c^\infty(N; TN)$. Since $\eta$ has compact support and $u_k$ is a c.w.d. map, $\eta \circ u_k$ is c.w.d. as well. Hence, by Propositions \ref{pr:product_map_derivative} and \ref{pr:lipschitz_composition} the real valued function $(\nu_k, \eta \circ u_k)_h$ is weakly differentiable. Furthermore, if we denote the vector field $\frac{\partial}{\partial x_j}$ on $Q$ by $X_j$, then we have
\[
\partial_{x_j} (\nu_k, \eta \circ u_k)_h = (K_{TN} \circ d\nu_k \circ X_j, \eta \circ u_k)_h + (\nu_k, K_{TN} \circ d(\eta \circ u_k) \circ X_j)_h \quad \text{a.e. in } Q.
\]
Let $\zeta \in C_c^\infty(Q)$. Differentiating $(\nu_k, \eta \circ u_k)_h\zeta$ with respect to $x_j$ and using the divergence theorem, we obtain
\begin{equation} \label{eq:integration_by_parts_with_big_terms}
    \int_Q (K_{TN} \circ d\nu_k \circ X_j, \eta \circ u_k)_h \zeta \, dx = -\int_Q (\nu_k, K_{TN} \circ d(\eta \circ u_k) \circ X_j)_h \zeta \, dx + \int_Q (\nu_k, \eta \circ u_k)_h \partial_{x_j}\zeta \, dx.
\end{equation}
Thanks to \eqref{eq:lp_and_pointwise_convergence} and the convergence $u_k \to u$ in $\dot{W}^{1,p}(Q; N)$, we see that the limit of the right-hand side is
\begin{multline*}
-\int_Q (\nu, K_{TN} \circ d(\eta \circ u) \circ X_j)_h \zeta \, dx + \int_Q (\nu, \eta \circ u)_h \partial_{x_j}\zeta \, dx \\
= \int_Q (K_{TN} \circ d\nu \circ X_j, \eta \circ u)_h \zeta \, dx = \int_Q (du \circ S_u \circ X_j, \eta \circ u)_h \zeta \, dx,
\end{multline*}
where we used the dominated convergence theorem to pass to the limit. On the other hand, we can use \eqref{eq:annoying_convergence} and the dominated convergence theorem to pass to the limit on the left-hand side of \eqref{eq:integration_by_parts_with_big_terms}. This yields
\begin{multline*}
    \int_Q (K_{TN} \circ d\nu_k \circ X_j, \eta \circ u_k)_h \zeta \, dx = \int_Q (du_k \circ S_{u_k} \circ X_j, \eta \circ u_k)_h \zeta \, dx \to \int_Q (du \circ S \circ X_j, \eta \circ u)_h \zeta \, dx.
\end{multline*}
Consequently, we arrive at
\[
\int_Q (du \circ S \circ X_j, \eta \circ u)_h \zeta \, dx = \int_Q (du \circ S_u \circ X_j, \eta \circ u)_h \zeta \, dx
\]
for all $\eta \in C_c^\infty(N; TN)$, $\zeta \in C_c^\infty(Q)$. It is now immediate that $du \circ S = du \circ S_u$ a.e. in $Q$.

\section*{Acknowledgements}

This work is funded by the Deutsche Forschungsgemeinschaft (DFG, German Research Foundation) under Germany's Excellence Strategy EXC 2044–390685587, Mathematics Münster: Dynamics–Geometry–Structure. This paper is an outgrowth of my Master's thesis at the University of Bonn. I would like to thank Stefan Müller for proposing the problem and for insightful ideas that shaped the direction of this work. I would also like to thank Kerrek Stinson for helpful discussions.

\bibliography{references}
\bibliographystyle{plain}

\end{document}